\documentclass[12pt,reqno,a4paper]{amsart}%
\usepackage[T1]{fontenc}
\usepackage[utf8]{inputenc}

\usepackage{graphicx}%
\usepackage[centering]{geometry}
\usepackage{amsmath,amssymb,amsfonts}%
\usepackage{amsthm}%
\usepackage{mathrsfs}%
\usepackage[title]{appendix}%

\usepackage[dvipsnames]{xcolor}
\usepackage{hyperref}
\hypersetup{
colorlinks=true,
linkcolor=BlueViolet,
filecolor=magenta,
menucolor=red,
urlcolor=cyan,
citecolor=BrickRed,
}
\usepackage{biblatex}
\theoremstyle{plain}%
\newtheorem{theorem}{Theorem}[section]%
\newtheorem*{thm*}{Theorem}%
\newtheorem{thmx}{Theorem}%
\newtheorem{thmxx}{Theorem}%
\newtheorem{proposition}[theorem]{Proposition}%
\newtheorem{lemma}[theorem]{Lemma}%

\theoremstyle{remark}%
\newtheorem{remark}[theorem]{Remark}%
\newtheorem{corollary}[theorem]{Corollary}%

\theoremstyle{definition}%
\newtheorem{definition}[theorem]{Definition}%

\newcommand{\C}{\mathbb{C}}
\newcommand{\R}{\mathbb{R}}
\renewcommand{\S}{\mathbb{S}}
\newcommand{\dK}{\partial\!K}
\DeclareMathOperator{\dmesure}{d}
\newcommand{\dx}{\dmesure\!}
\DeclareMathOperator{\trace}{trace}
\DeclareMathOperator{\Span}{span}
\DeclareMathOperator{\End}{End}
\DeclareMathOperator{\Id}{Id}
\newcommand{\delr}{\partial_r}
\newcommand{\Jdelr}{J\!\delr}
\newcommand{\Ndr}{\nabla_{\delr}}
\newcommand{\Ldr}{\mathcal{L}_{\delr}}
\newcommand{\JE}{J\!E}
\renewcommand{\bar}{\overline}
\renewcommand{\tilde}{\widetilde}
\numberwithin{equation}{subsection}
\newcommand{\ALCH}{(\hyperlink{ALCH}{ALCH})~}
\newcommand{\ALS}{(\hyperlink{ALS}{ALS})~}

\title[Asymptotic strictly pseudoconvex CR structure for ALCH manifolds]{Asymptotic strictly pseudoconvex CR structure for asymptotically locally complex hyperbolic manifolds}
\author{Alan Pinoy}

\begin{document}

\begin{abstract}
In this paper, we build a compactification by a strictly pseudoconvex CR structure for a complete and non-compact K\"ahler manifold whose curvature tensor is asymptotic to that of the complex hyperbolic space.
To do so, we study in depth the evolution of various geometric objects that are defined on the leaves of some foliation of the complement of a suitable convex subset, called an \textit{essential subset}, whose leaves are the equidistant hypersurfaces above this latter subset.
With a suitable renormalization which is closely related to the anisotropic nature of the ambient geometry, the above mentioned geometric objects converge near infinity, inducing the claimed structure on the boundary at infinity.
\end{abstract}



\maketitle


\section{Introduction}
\label{section:intro}

The study of the asymptotic geometry of complete non-compact Riemannian manifolds have proven to be fruitful in the understanding of the geometry of complex domains, driven by the following remark.
By endowing the interior of a bounded domain with a complete metric, which then sends its boundary to infinity, one can read much information on the geometry of the boundary in the asymptotic development of the metric, see \cite{fefferman_bergman_1974,fefferman_monge-ampere_1976,hirachi_construction_2000}.
The induced geometric structure on the boundary leads to geometric invariants of the domain.
The Bergman metric and the K\"ahler-Einstein metric are examples of such metrics, and have been at the centre of complex geometry for decades.
On the unit ball of $\C^n$, these two latter metrics are equal up to a multiplicative constant.
This example is particularly interesting, and is called the complex hyperbolic space:
it is the unique simply connected and complete K\"ahler manifold with constant holomorphic sectional curvature equal to $-1$.
In that sense, it is the complex counterpart of the real hyperbolic space.
In polar coordinates in an exponential chart, the complex hyperbolic metric is given by the following expression
\begin{equation*}
  \dx r^2 + 4 \sinh^2(r)\theta\otimes \theta + 4\sinh\left(\frac{r}{2}\right) \gamma,
\end{equation*}
where $\theta$ is the standard contact form of the unit sphere of $\C^n$, and $\gamma = \dx \theta (\cdot,i\cdot)$ is the Levi form induced on the contact distribution $\ker \theta$.
This form of the metric reveals the geometric structure of the odd dimensional sphere, which is a strictly pseudoconvex Cauchy-Riemann (CR) manifold.
We say that the CR sphere is the sphere at infinity of the complex hyperbolic sphere.
This example is a particular case of the more general non-compact rank one symmetric spaces, namely the real hyperbolic spaces, the complex hyperbolic spaces, the quaternionic hyperbolic spaces and the octonionic hyperbolic plane.
They all admit a sphere at infinity, which is endowed with a particular geometric structure closely related to the Riemannian metric of these spaces; they are called conformal infinities.
Asymptotically hyperbolic Einstein metrics with prescribed conformal infinity have been built by C. R. \textsc{\small Graham} and J. M. \textsc{\small Lee} \cite{graham_einstein_1991,lee_fredholm_2006}.
The general case of asymptotically symmetric metrics have been covered by O. \textsc{\small Biquard} \cite{biquard_metriques_2000}.
Similar results, yet unpublished, has been obtained simultaneously by J. \textsc{\small Roth} in his Ph.D thesis for the complex hyperbolic case \cite{roth_perturbation_1999}.

In a series of papers \cite{bahuaud_intrinsic_2008,bahuaud_conformal_2011,bahuaud_holder_2008,gicquaud_conformal_2013}, E. \textsc{\small Bahuaud}, R. \textsc{\small Gicquaud} and T. \textsc{\small Marsh} have studied a converse problem in the real hyperbolic setting.
They give therein intrinsic geometric conditions under which a complete non-compact Riemannian manifold whose sectional curvature decays sufficiently fast to $-1$ near infinity have a conformal boundary at infinity, modelled on that of the hyperbolic space.
This generalizes a work of M. \textsc{\small Anderson} and R. \textsc{\small Schoen} \cite{anderson_positive_1985}.
These geometric conditions are the existence of a convex core, called \textit{essential subset}, and the convergence near infinity of the sectional curvature to $-1$, and of some covariant derivatives of the curvature tensor to $0$, with exponential decays.

By trying to determine which complete K\"ahler manifolds arise as bounded submanifolds of some Hermitian space $\C^N$, J. \textsc{\small Bland} has studied an analogous problem in the K\"ahler case, and has built a compactification by a CR structure for some complete non-compact manifolds whose curvature tensor is asymptotic to that of the complex hyperbolic space in some particular coordinates \cite{bland_existence_1985,bland_bounded_1989}.
Nonetheless, his assumptions appear to be rather restrictive, and not completely geometric: roughly, it is asked that the curvature tensor, with a suitable renormalization, extends up to the boundary in some already compactified coordinates, with high regularity.
They imply, in particular, the \textit{a posteriori} estimates $R = R^0 + \mathcal{O}(e^{-3r})$ and $\nabla R = \mathcal{O}(e^{-4r})$, where $R^0$ is the constant $-1$ holomorphic sectional curvature tensor, and $r$ is the distance function to some compact subset.
However, O. \textsc{\small Biquard} and M. \textsc{\small Herzlich} have shown in \cite{biquard_burns-epstein_2005} that, in real dimension $4$, the curvature tensor $R$ of an asymptotically complex hyperbolic Einstein metric has the following asymptotic development
\begin{equation*}
  R = R^0 + C e^{-2r} + \mathcal{O}(e^{-2r}),
\end{equation*}
where $C$ turns out to be a (non-zero) multiple of the Cartan tensor of the CR structure at infinity.
Since the Cartan tensor of a CR structure vanishes exactly when the CR structure is spherical (that is, if it is locally CR diffeomorphic to the unit sphere), it seems that J. \textsc{\small Bland}'s results only hold for a few examples of asymptotically complex hyperbolic metrics, at least in real dimension $4$ and in the Einstein setting.

More recently, it has been shown by F. \textsc{\small Bracci}, H. \textsc{\small Gaussier} and A. \textsc{\small Zimmer} that a convex domain in $\C^n$ with boundary of class at least $\mathcal{C}^{2,\alpha}$ ($\alpha >0$) is strictly pseudoconvex if and only if one can endow this domain with a complete K\"ahler metric whose holomorphic sectional curvature has range in $[-1-\varepsilon,-1+\varepsilon]$ near the boundary \cite{bracci_geometry_2018,zimmer_gap_2018}.
Here, the constant $\varepsilon$ only depends on the dimension and the regularity of the boundary.
In this case, the holomorphic sectional curvature has the form $-1 + \mathcal{O}(e^{-ar})$ for some $a>0$, where $r$ is the distance function from a compact subset.
Let us mention that this result has been proven to be false if the regularity of the boundary is only $\mathcal{C}^{2}$ \cite{fornaess_non-strictly_2018}.

Inspired by the work of E. \textsc{\small Bahuaud}, R. \textsc{\small Gicquaud} and T. \textsc{\small Marsh}, we give in this paper a geometric characterization of complete non-compact K\"ahler manifolds admitting a compactification by a strictly pseudoconvex CR structure.
The study is more intricate in the complex hyperbolic setting than in the real hyperbolic one, due to the anisotropic nature of complex hyperbolic geometry.
In contrast with J. \textsc{\small Bland}'s results, our study only relies on purely geometric assumptions, and requires a control of the curvature and its first covariant derivative to an order strictly less than $\mathcal{O}(e^{-2r})$.
We consider a complete non-compact K\"ahler manifold $(M,g,J)$ whose Riemann curvature tensor $R$ is asymptotic to $R^0$, the curvature tensor of constant holomorphic sectional curvature $-1$.
We first prove that if $R$ is exponentially close to $R^0$ near infinity, then the metric tensor has an asymptotic development at infinity that is similar to that of the complex hyperbolic space.

\begin{thmxx} \label{thmxx:A}
Let $(M,g,J)$ be a complete non-compact K\"ahler manifold with an essential subset $K$.
Assume that the sectional curvature of $\bar{M\setminus K}$ is negative and that there exists $a>1$ such that $\|R-R^0\|_g = \mathcal{O}(e^{-ar})$.
Then $\dK$ is endowed with a nowhere vanishing continuous differential 1-form $\eta$ and a continuous field of symmetric positive semi-definite bilinear forms $\gamma_H$, positive definite on the distribution $H=\ker\eta$, such that the metric reads in an exponential chart
\begin{equation*}
\dx r^2 + e^{2r} \eta \otimes \eta + e^r \gamma_H + \text{lower order terms}.
\end{equation*}
\end{thmxx}

See section \ref{section:thmA} for a precise statement.
In view of the computations, and due to the characterization of strictly pseudoconvex domains by F. \textsc{\small Bracci}, H. \textsc{\small Gaussier} and A. \textsc{\small Zimmer}, the decay rate of the curvature is a reasonable assumption.
Such a K\"ahler manifold will be referred to as an asymptotically locally complex hyperbolic manifold.
The local condition on the curvature tensor is thus a sufficient condition to recover a development of the metric similar to that of the model space.
The differential form $\eta$ is called the canonical 1-form at infinity, and the field of symmetric bilinear forms $\gamma_H$ is called the Carnot-Carathéodory metric at infinity.
Under the extra condition that the metric is asymptotically locally symmetric, we can moreover prove the following.

\begin{thmxx}\label{thmxx:B}
Let $(M,g,J)$ be a complete non-compact K\"ahler manifold satisfying the assumptions of Theorem \ref{thmxx:A}.
Assume furthermore that there exists $b >1$ such that $\|\nabla R\|_g = \mathcal{O}(e^{-br})$.
Then the canonical 1-form at infinity $\eta$ is a contact form of class $\mathcal{C}^1$.
\end{thmxx}

Under higher exponential decays, the Carnot-Carathéodory metric also gains one order of regularity.

\begin{thmxx}\label{thmxx:C}
Let $(M,g,J)$ be a complete non-compact K\"ahler manifold satisfying the assumptions of Theorem \ref{thmxx:B}.
Assume furthermore that $\min\{a,b\}>\frac{3}{2}$.
Then the Carnot-Carathéodory metric at infinity $\gamma_H$ is of class $\mathcal{C}^1$.
\end{thmxx}

We then show that under the assumptions of Theorem \ref{thmxx:C}, the contact distribution $H$ is endowed with a canonical integrable almost complex structure $J_H$ having the same properties as that of the model space.

\begin{thmxx}\label{thmxx:D}
Let $(M,g,J)$ be a complete non-compact K\"ahler manifold satisfying the assumptions of Theorem \ref{thmxx:C}.
Then the contact distribution at infinity $H=\ker \eta$ is endowed with an integrable almost complex structure $J_H$ at infinity of class $\mathcal{C}^1$ such that the Carnot-Carathéodory metric at infinity is given by $\gamma_H = \dx \eta(\cdot,J_H\cdot)$.
In particular, $(\dK, H,J_H)$ is a strictly pseudoconvex CR manifold of class $\mathcal{C}^1$.
\end{thmxx}

\subsection*{Structure of the paper}
\label{subsection:structure}

In section \ref{section:preliminaries}, we detail the notations and the setting.
In section \ref{section:ALCH_manifolds}, we define the \textit{asymptotically locally complex hyperbolic} \ALCH and \textit{asymptotically locally symmetric} \ALS conditions and give a lower bound on the volume growth.
In section \ref{section:normal_jacobi}, we study the asymptotic behaviour of normal Jacobi fields, and then prove Theorem \ref{thmxx:A}.
Section \ref{section:contact_structure} is devoted to the proof of Theorems \ref{thmxx:B} and \ref{thmxx:C}.
Section \ref{section:almost_complex_structure} is finally dedicated to the definition of the almost complex structure at infinity and to the proof of Theorem \ref{thmxx:D}.
Useful curvature computations can be found in the Appendix \ref{appendix:A}.

\subsection*{Acknowledgments}
This research was conducted while the author was pursuing his Ph.D thesis at the Institut Montpelli\'erain Alexander Grothendieck, Universit\'e de Montpellier, France.
The author would like to thank Marc Herzlich and Philippe Castillon for suggesting this problem to him and for their careful guidance.
May Guillaume Ferri\` ere be thanked as well for his early analytical remarks.
This article greatly benefited from the relevant comments of Gilles Carron and Jack Lee, and the author would like to express his gratitude to them.
Finally, the author would like to thank the anonymous referee for their valuable comments that improved the quality and the clarity of the paper.

\section{Preliminaries}
\label{section:preliminaries}
\subsection{Notations}
\label{subsec:notations}
Let $(M^{2n+2},g,J)$ be a complete, non-compact K\"ahler manifold of real dimension $2n+2$, $n\geqslant 1$.
We denote by $\nabla$ its Levi-Civita connection.
We recall that $(M,g,J)$ is K\"ahler if the almost complex structure $J$ is parallel and satisfies $g(JX,JY) = g(X,Y)$ for all tangent vectors $X$ and $Y$.
The Riemann curvature tensor $R$ is defined by
\begin{equation}\label{eq:convention_curv}
R(X,Y)Z = \nabla_{[X,Y]}Z - \big(\nabla_X(\nabla_ZY)-\nabla_Y(\nabla_XZ)\big).
\end{equation}
By abuse of notation, its four times covariant version is still denoted by $R$, that is $R(X,Y,Z,T) = g(R(X,Y)Z,T)$.
Please note that our convention is that of \cite{besse_einstein_2007,gallot_riemannian_2004}, which is opposite to that of \cite{do_carmo_riemannian_1992,lee_introduction_2019}.
In our case, the sectional curvature of a linear plane $P$ with orthonormal basis $\{X,Y\}$ is $\sec (P) = \sec(X,Y)=R(X,Y,X,Y)$ and the holomorphic sectional curvature is given by $R(X,JX,X,JX)$.

Let $K\subset M$ be a compact codimension $0$ submanifold with hypersurface boundary $\dK$ oriented by a unit normal $\nu$.
The associated outward normal exponential map $E\colon \R_+ \times \dK \to M$ is defined by $E(r,p) = \gamma_p(r)$, where $\gamma_p$ is the unit speed geodesic with initial data $\gamma_p(0) = 0$ and $\gamma_p'(0) = \nu(p)$.
Following \cite{bahuaud_holder_2008}, the submanifold $K$ is called an \textit{essential subset} if $\dK$ is convex with respect to $\nu$, meaning that its shape operator is non-negative, and if $E\colon \R_+ \times \dK \to \bar{M\setminus K}$ is a diffeomorphism.
If $K$ is totally convex and if the sectional curvature outside of $K$ is negative, then $K$ is an essential subset (see {\cite[Theorem 3.1]{bahuaud_holder_2008}}).
Moreover, the visual boundary of $(M,g)$ is homeomorphic to the boundary of $K$.

Assume that $K\subset M$ is an essential subset.
The \textit{radial vector field} $\delr$ is the vector field on $\bar{M\setminus K}$ defined by $\delr(\gamma_p(r)) = E(r,p)_*\frac{\dx}{\dx r}= \gamma_p'(r)$.
A tensor defined on $\bar{M\setminus K}$ is said to be \textit{radially parallel} if its covariant derivative in the $\delr$ direction vanishes.
Since $(\gamma_p)_{p\in \dK}$ are geodesics, $\delr$ is radially parallel.
The shape operator $S$ and the Jacobi operator $R_{\delr}$ are the fields of symmetric endomorphisms defined by $SX = \nabla_X\delr$ and $R_{\delr}X = R(\delr,X)\delr$.
They are related by the Riccati equation
\begin{equation}\label{eq:riccati}
\Ldr S = \Ndr S = - S^2 - R_{\delr}.
\end{equation}
They both vanish on $\delr$ and take values in $\{\delr\}^{\perp}$.
From the Riccati equation we derive the following: if $\dK$ is convex and the sectional curvature of $\bar{M\setminus K}$ is non-positive, then the level hypersurfaces above $\dK$ are convex.

If $v \in T_p\dK$ is a tangent vector, the associated \textit{normal Jacobi field} along $\gamma_p$ is defined as $Y_v(r) = E(r,p)_*v$.
If $v$ is a (local) vector field on $\dK$, we write $Y_v(\gamma_p(r))= Y_{v(p)}(r)$, which is still called a normal Jacobi field.
\begin{lemma}
\label{lem:Y_v}
If $v$ is a (local) vector field on $\dK$, then:
\begin{enumerate}
\item The vector fields $Y_v$ and $\delr$ commute,
\item $\Ndr Y_v = SY_v$,
\item $Y_v$ is everywhere orthogonal to $\delr$,
\item The restriction of $Y_v$ along any radial geodesic $\gamma_p$ is a Jacobi field.
\end{enumerate}
\end{lemma}
\begin{proof}
\begin{enumerate}
\item Since $E$ is a diffeomorphism onto its image, the naturality of the Lie bracket yields $[\delr,Y_v] = E_*[\frac{\dx}{\dx r},v] = E_* 0 = 0$.
\item The Levi-Civita connection is torsion-free, hence $SY_v - \Ndr Y_v =[\delr,Y_v]$.
The result follows by the first point.
\item The equation of geodesics $\nabla_{\delr}\delr = 0$ together with the second point give $\delr (g(Y_v,\delr)) = g(SY_v,\delr)$.
This last expression identically vanishes since $S$ has range in $\{\delr\}^{\perp}$.
Hence, the function $g(Y_v,\delr)$ is constant along radial geodesics, and the result follows from the initial condition $g(v,\nu) = 0$.
\item From the second point, it holds that $\Ndr Y_v = SY_v$.
One thus obtains the equality $\Ndr (\Ndr Y_v) =\Ndr (SY_v) = (\Ndr S)Y_v + S \Ndr Y_v$, and the result follows from the Riccati equation.
\end{enumerate}
\end{proof}

Recall that $(M,g,J)$ is assumed to be K\"ahler.
We define the \textit{constant $-1$ holomorphic sectional curvature tensor} $R^0$ by
\begin{equation*}\label{eq:def_R0}
\begin{split}
R^0(X,Y,Z,T) &= \frac{1}{4}\big( g(X,T)g(Y,Z) - g(X,Z)g(Y,T)\\
& \quad + g(X,JT)g(Y,JZ) - g(X,JZ)g(Y,JT)\\
& \quad + 2g(X,JY)g(T,JZ)\big).
\end{split}
\end{equation*}
It is a parallel tensor having the symmetries of a curvature tensor.
In case $(M,g,J)$ is the complex hyperbolic space, then $R=R^0$.
Note that this definition differs from that of \cite[section IX.7]{kobayashi_foundations_1996-1} by its sign: the reason is that we require for $R^0$ to have $-1$ holomorphic sectional curvature.
For any orthonormal pair of vectors fields $\{X,Y\}$, $R^0(X,Y,X,Y) = -\frac{1}{4}(1 + 3g(JX,Y)^2)$, from which is deduced the fundamental pinching
\begin{equation*}\label{eq:R0_pinched}
-1 \leqslant R^0(X,Y,X,Y) \leqslant - \frac{1}{4}.
\end{equation*}
If $P = \Span\{X,Y\}$, the lower bound is achieved exactly when $P = JP$ (we say that $P$ is a \textit{complex line}) while the upper bound is achieved exactly when $P\perp JP$ (we say that $P$ is a \textit{totally real plane}).

\subsection{Contact and CR manifolds}

\label{section:CR}
A \emph{contact form} on an $(2n+1)$-dimensional manifold $M$ is a differential form of degree one $\alpha$ such that $\alpha\wedge \dx \alpha^n \neq 0$ everywhere.
The associated contact structure is $H=\ker \alpha$.
We say that $(M,H)$ is a \emph{contact manifold}, and that $\alpha$ calibrates $H$.
Note that for $H$ fixed, $\alpha$ is not unique.
The \emph{Reeb vector field} of $\alpha$ is the unique vector field $X_{\alpha}$ such that $\alpha(X_{\alpha}) = 1$ and $\dx\alpha(X_{\alpha},\cdot)=0$.

An \emph{almost complex structure} $J$ on $(M,H)$ is a section of the bundle $\End (H)$ such that $J^2 = -\Id_H$.
The complexified bundle $H\otimes \C$ splits as $H \otimes \C = H^{1,0}\oplus H^{0,1}$ into the eigenspaces of the complex linear extension of $J$, where $H^{1,0} = \{X-iJX \mid X \in H\}$ and $H^{0,1} = \{X+iJX \mid X \in H\}$.
The almost complex structure $J$ is said to be \emph{integrable} if the sections of $H^{1,0}$ form a linear subalgebra of the sections of $TM\otimes \C$, that is if $[X,Y] \in \Gamma(H^{1,0})$ whenever $X,Y\in \Gamma(H^{1,0})$.

A \textit{CR manifold} $(M,H,J)$ is a contact manifold $(M,H)$ endowed with an integrable almost complex structure on $H$.
When $H = \ker \alpha$ and $\dx \alpha (\cdot,J\cdot)$ is positive definite on $H$, $(M,\alpha,J)$ is called a \textit{strictly pseudoconvex CR manifold}.
CR geometry is the natural geometry of real hypersurfaces in complex manifolds.
The toy model is $\S$, the unit sphere of $\C^n$, with contact distribution $H = T\S \cap (iT\S)$ and with almost complex structure $J$ given by the multiplication by $i$ in the fibres.
Endowed with its natural contact form, it is a strictly pseudoconvex CR manifold.

\subsection{Some analysis results}

Throughout the paper, we will use the following two theorems from real analysis several times.

\begin{thm*}[\textsc{\small Gr\"onwall}'s inequality]\label{thm:Gronwall}
Let $I=[t_0,T) \subset \R$ be an interval with $T\in (t_0,+\infty]$ and let $\varphi,\alpha$ and $\beta \colon I \to \R$ be respectively continuous, non-decreasing, and non-negative continuous functions.
Assume that
\begin{equation}\label{eq:ref_gronwall_type}
\forall t \in I,\quad \varphi(t) \leqslant \alpha(t) + \int_{t_0}^t \beta(s)\varphi(s) \dx s.
\end{equation}
Then it holds that
\begin{equation*}
\forall t \in I,\quad \varphi(t) \leqslant \alpha(t) \exp\left(\int_{t_0}^t \beta(s)\dx s\right).
\end{equation*}
\end{thm*}
An inequality of the form \eqref{eq:ref_gronwall_type} will be referred to as a  ``Gr\"onwall-like inequality''.

\begin{thm*}[\textsc{\small Rademacher}'s Theorem]\label{thm:Rademacher}
Let $\varphi \colon I \to \R$ be a locally Lipschitz function defined on an interval $I$, and let $t_0\in I$ be fixed.
Then $\varphi$ is almost everywhere differentiable, and it holds that
\begin{equation*}
\forall t \in I,\quad \varphi(t) = \varphi(t_0) + \int_{t_0}^t \varphi'(s)\dx s,
\end{equation*}
where the latter integral is the Lebesgue integral of the almost everywhere defined function $\varphi'$.
\end{thm*}
Rademacher's Theorem will be applied to the restriction of the norm of tensors restricted along radial geodesics $\gamma_p$.

The following Lemma will be useful in order to estimate the growth of normal Jacobi fields.

\begin{lemma}\label{lem:equadiff_lemma}
Let $f \colon \R_+ \to \R$ be a $\mathcal{C}^2$ function.
Assume that there exist three positive constants $\alpha, \beta$ and $\gamma$ such that
\[
\forall t \geqslant0,\quad |f''(t) + \alpha f'(t)| \leqslant \gamma e^{-\beta t}.
\]
Then $f$ has a limit $f_{\infty}\in \R$ when $t \to +\infty$ and it holds that
\begin{equation*}
\begin{split}
\forall t \geqslant 0,\quad |f'(t)| & \leqslant \begin{cases}
\left(|f'(0)| + \frac{\gamma}{\beta-\alpha}\right)e^{-\beta t} & \text{if } \beta < \alpha, \\
\left(|f'(0)| + \gamma t \right) e^{-\alpha t} & \text{if } \beta = \alpha, \\
\left(|f'(0)| + \frac{\gamma}{\beta-\alpha} \right)e^{-\alpha t} & \text{if } \beta > \alpha,
\end{cases}\\
\forall t \geqslant 0,\quad |f(t)-f_{\infty}| & \leqslant \begin{cases}
\left(|f'(0)| + \frac{\gamma}{\beta-\alpha}\right)\frac{e^{-\beta t}}{\beta} & \text{if } \beta < \alpha, \\
\frac{\alpha|f'(0)| + \gamma(\alpha t+1)}{\alpha^2}  e^{-\alpha t} & \text{if } \beta = \alpha, \\
\left(|f'(0)| + \frac{\gamma}{\beta-\alpha} \right)\frac{e^{-\alpha t}}{\alpha} & \text{if } \beta > \alpha.
\end{cases}
\end{split}
\end{equation*}
\end{lemma}
\begin{proof}
Write $e^{\alpha t}(f''(t)+\alpha f'(t)) = \left(e^{\alpha t} f'(t)\right)'$.
Integrating on $[0,t]$ and using the assumption yields
\begin{equation}
\label{eq:lem1.2}
\forall t \geqslant 0, \quad |f'(t)| \leqslant e^{-\alpha t} |f'(0)| + \gamma e^{-\alpha t}\int_0^t e^{(\alpha-\beta)x} \dx x.
\end{equation}
There are three cases to consider depending on $\beta < \alpha$, $\beta=\alpha$ or $\beta > \alpha$.
In any case, equation \eqref{eq:lem1.2} shows that $f'$ is integrable and hence that $f$ converges.
The result follows from a straightforward integration.
\end{proof}

We finally give a Lemma that will be applied to give a uniform bound on the norm of the shape operator.
\begin{lemma}\label{lem:bound_eigenvalue}
Let $\sigma \colon 	\R_+ \to \R$ be a locally Lipschitz function.
Let $C$ and $a$ be positive constants such that the following inequality holds almost everywhere
\begin{equation*}
\sigma' \leqslant -\sigma^2 + 1 + C e^{-at}.
\end{equation*}
Then there exists a constant $C'>0$ depending only on $C$ and $a$ such that
\begin{equation*}
\forall t \geqslant 0,\quad \sigma(t) \leqslant 1 + (\sigma(0)+C') \begin{cases}
e^{-at} & \text{if } 0<a<2,\\
(t+1)e^{-2t} & \text{if } a =2,\\
e^{-2t}  & \text{if } a >2.
\end{cases}
\end{equation*}
\end{lemma}

\begin{proof}
Let $\varsigma = \sigma - 1$.
Then $\varsigma$ is locally Lipschitz and almost everywhere differentiable, and it follows from the assumption on $\sigma$ and from the fact that $\varsigma ^2 \geqslant 0$, that
\begin{equation*}
\varsigma'(t) \leqslant -2 \varsigma + Ce^{-at}  \quad a.e.
\end{equation*}
Thus, it holds that
\begin{equation*}
(\varsigma(t)e^{2t})' \leqslant Ce^{(2-a)t} \quad a.e.
\end{equation*}
Integrating this last inequality yields
\begin{equation*}
\forall t \geqslant 0,\quad \varsigma(t) \leqslant \varsigma(0)e^{-2t} + C\begin{cases}
\frac{e^{-at}-e^{-2t}}{2-a} & \text{if } 0<a<2,\\
t e^{-2t} & \text{if } a = 2,\\
\frac{e^{-2t}-e^{-at}}{a-2}& \text{if } a >2.
\end{cases}
\end{equation*}
The result now follows from comparing the exponents.
\end{proof}


\section[Asymptotically locally complex hyperbolic manifolds]{Asymptotically locally complex hyperbolic manifolds}
\label{section:ALCH_manifolds}
\subsection{The ALCH and ALS conditions}

We define the two following asymptotic conditions.
\begin{definition}[\ALCH and \ALS manifolds]
Let $(M,g,J)$ be a complete non-compact K\"ahler manifold, $K\subset M$ be a compact subset and $r = d_g(\cdot,K)$ be the distance function to $K$.
\begin{enumerate}
\item $(M,g,J)$ is said to be \hypertarget{ALCH}{\textit{asymptotically locally complex hyperbolic}}, \ALCH in short, of order $a>0$, if there exists a constant $C_0>0$ such that
\begin{equation*}
\|R-R^0\|_g \leqslant C_0 e^{-ar}.
\end{equation*}
\item $(M,g)$ is said to be \hypertarget{ALS}{\textit{asymptotically locally symmetric}}, \ALS in short, of order $b>0$, if there exists a constant $C_1 >0$ such that
\begin{equation*}
\|\nabla R\|_g \leqslant C_1 e^{-br}.
\end{equation*}
\end{enumerate}
\end{definition}
\begin{remark}
These two definitions do not depend on the choice of the compact subset $K \subset M$.
\end{remark}
In practice, $K$ will refer to an essential subset.
The complex hyperbolic space is of course \ALCH and \ALS of any order since in that case, $R=R^0$ and $\nabla R=0$.

\subsection{First consequences}

We fix $(M,g,J)$ an \ALCH manifold of order $a>0$.
The following Lemmas are direct consequences of the definition of $R^0$.
\begin{lemma}\label{lem:R_bounded}
The norm of the Riemann curvature tensor of an \ALCH manifold of order $a>0$ is uniformly bounded.
\end{lemma}

\begin{lemma}\label{lem:pinched_sectional_curvature}
Let $p\in M$ and $P\subset T_pM$ be a tangent plane.
Then
\begin{equation*}
-1 - C_0 e^{-ar(p)} \leqslant \sec (P) \leqslant -\frac{1}{4} + C_0e^{-ar(p)}.
\end{equation*}
\end{lemma}

We now assume that $K\subset M$ is an essential subset and that the sectional curvature of $\bar{M\setminus K}$ is non-positive.
In that case, the shape operator $S$ is positive semi-definite on $\bar{M\setminus K}$, and its operator norm at a point $\gamma_p(r)\in \bar{M\setminus K}$ is given by its largest eigenvalue.
\begin{proposition}\label{prop:norm_S_uniformly_bounded}
There exists a constant $C>0$ independent of $(r,p)$ such that
\begin{equation*}
\|S_{\gamma_p(r)}\|_g \leqslant 1 + C
\begin{cases}
e^{-ar} & \text{if } 0<a<2,\\
(r+1)e^{-2r} & \text{if } a =2,\\
e^{-2r}  & \text{if } a >2.
\end{cases}
\end{equation*}
In particular, $\|S\|_g$ is uniformly bounded on $\bar{M\setminus K}$, and $\|S\|_g-1$ is bounded above by an integrable function on $\bar{M\setminus K}$.
\end{proposition}
\begin{proof}
Let $p \in \dK$ and $\sigma \colon \R \to \R$ be defined by $\sigma(r) = \|S_{\gamma_p(r)}\|_g$.
Since $S$ is positive semi-definite, $\sigma(r)$ is the largest eigenvalue of $S_{\gamma_p(r)}$.
Identify all tangent spaces along $\gamma_p$ using parallel transport, and let $(S_r)_{r\geqslant 0}$ and $(R_r)_{r\geqslant 0}$ be the associated endomorphisms of $T_p\dK$ obtained from $S$ and $R_{\delr}$.
Therefore, $\sigma$ is given by the expression $\sigma(r) = \sup_{X\in T_p\dK, \|X\|_g=1} g(S_rX,X)$.
In this identification, the Riccati equation reads $S_r' = -S_r^2 - R_r$.
Let $X\in T_p\dK$ be a unit vector, $r\geqslant 0$ and $h>0$.
Then
\begin{equation*}
\begin{split}
g(S_{r+h}X,X) &= g\big((S_r + h S_r' + o(h))X,X\big)  \\
&= g\big((S_r-hS_r^2)X,X\big) + h (-g(R_rX,X) + o(1)\big).
\end{split}
\end{equation*}
For small enough $h>0$, $S_r - hS_r^2$ is positive semi-definite and its largest eigenvalue is given by $\sigma(r) - h\sigma(r)^2$.
In addition, $g(R_rX,X)$ is the sectional curvature of the tangent plane spanned by $\delr$ and the parallel transport of $X$ along $\gamma_p$ evaluated at $\gamma_p(r)$, which is bounded below by $-1-C_0e^{-ar}$ (Lemma \ref{lem:pinched_sectional_curvature}).
It then follows that
\begin{equation*}
g(S_{r+h}X,X) \leqslant \sigma(r) - h\sigma(r)^2 + h\left(1+C_0e^{-ar} + o(1)\right).
\end{equation*}
Taking the supremum over all unit vectors $X$ of this last inequality yields
\begin{equation*}
\sigma(r+h) \leqslant \sigma(r) - h\sigma(r)^2 + h\left(1+C_0e^{-ar} + o(1)\right),
\end{equation*}
from which is deduced the inequality
\begin{equation}\label{eq:limsup}
\limsup_{h\to 0} \frac{\sigma(r+h) - \sigma(r)}{h} \leqslant - \sigma(r)^2 + 1 + C_0e^{-ar}.
\end{equation}
Since $S$ is smooth and $\|\cdot\|$ is Lipschitz, $\sigma$ is locally Lipschitz, and by Rademacher's Theorem, $\sigma$ is almost everywhere differentiable.
It then follows from equation \eqref{eq:limsup} that
\begin{equation*}
\sigma ' \leqslant -\sigma^2 + 1 + C_0e^{-ar} \quad a.e.
\end{equation*}
According to Lemma \ref{lem:bound_eigenvalue}, there exists $C'>0$ depending only on $C_0$ and $a$ such that
\begin{equation*}
\|S_{\gamma_p(r)}\|_g \leqslant 1 + (\|S_p\|_g + C')\begin{cases}
e^{-ar} & \text{if } 0<a<2,\\
(r+1)e^{-2r} & \text{if } a =2,\\
e^{-2r}  & \text{if } a >2.
\end{cases}
\end{equation*}
The result now follows by defining $C = C' + \sup_{p\in \dK} \|S_p\|_g$ which is finite by compactness of $\dK$.
\end{proof}

\subsection{A volume growth lower bound}

The \ALCH assumption implies two bounds on the sectional curvature.
The lower bound $\sec \geqslant -1 + \mathcal{O}(e^{-ar})$ forces the largest eigenvalue of the shape operator, and thus its operator norm, to be uniformly bounded.
We shall now show that the upper bound $\sec \leqslant -\frac{1}{4} + \mathcal{O}(e^{-ar})$ forces the trace of the shape operator to be bounded from below.
We then derive a lower bound on the volume density function using equation \eqref{eq:volume_density_trace}.
Our study is inspired by \cite{cheeger_comparison_2008} and \cite{heintze_general_1978}.

We first show the following general Lemma from geometric comparison analysis.
It is very similar to different results stated in \cite{heintze_general_1978}.
The author could not find a precise proof of the exact formulation given below and therefore decided to give one, relying on classical Jacobi field techniques.
\begin{lemma}\label{lem:trace_bound_below}
Let $(M^{m+1},g)$ be a complete Riemannian manifold, $N\subset M$ be a hypersurface co-oriented by a unit normal $\nu$, $p\in \dK$, and $\kappa>0$ such that
\begin{enumerate}
\item $N$ is convex with respect to $\nu$,
\item there exists $R>0$ such that the normal exponential map $E$  is a diffeomorphism from $[0,R)\times N$ onto its image,
\item the sectional curvature on this image is non-positive,
\item there exists $r_0 \in [0,R)$ such for all $r\in [r_0,R)$, the sectional curvature of any linear plane tangent at $\gamma_p(r)$ is bounded above by $-\kappa^2$.
\end{enumerate}
Then it holds that
\begin{equation*}
\forall r \in [r_0,R),\quad \trace (S_{\gamma_p(r)}) \geqslant m\kappa \tanh \big(\kappa(r-r_0)\big).
\end{equation*}
\end{lemma}
\begin{proof}
The proof is quite long and goes in two steps.
We first bound from below the trace of $S$ by an integral related to the norm of some Jacobi fields.
Then, we compare the situation with its counterpart in the hyperbolic space of dimension $m+1$ and sectional curvature $-\kappa^2$, where $N$ is replaced by an isometrically embedded $m$ dimensional hyperbolic space of the same sectional curvature.

Let $\{\delr,E_1,\ldots,E_m\}$ be an orthonormal frame along $\gamma_p$ obtained using the parallel transport of an orthonormal basis $\{\nu,e_1,\ldots,e_m\}$ of $T_pM$.
Let $r \in [r_0,R)$ be fixed.
For $j \in \{1,\ldots,m\}$, let $Y_j$ be the unique Jacobi field along $\gamma_p$ such that $Y_j(r) = E_j(r)$ and $\Ndr Y_j(r) = S_{\gamma_p(r)}E_j(r)$.
It is shown similarly to the proof of Lemma \ref{lem:Y_v} that the equality $\Ndr Y_j = SY_j$ holds all along $\gamma_p$.
By assumptions 2. and 3. and by the convexity of $N$, $S$ is a positive definite operator and it follows that, for $j\in \{1,\ldots,m\}$
\begin{equation*}
g(SY_j(r),Y_j(r)) \geqslant g(SY_j(r),Y_j(r)) - g(SY_j(r_0),Y_j(r_0))
= \int_{r_0}^{r}\frac{\dx}{\dx t} g(SY_j,Y_j).
\end{equation*}
The Jacobi field equation yields the equality
\begin{equation*}
\frac{\dx}{\dx t}g(SY_j,Y_j) = \|\Ndr Y_j\|_g^2 - \sec(\delr,Y_j)\|Y_j\|_g^2.
\end{equation*}
Since $\trace \big(S_{\gamma_p(r)}\big) = \sum_{j=1}^m g(SE_j(r),E_j(r))$, the initial conditions for the Jacobi fields $Y_j$ and the assumption 4. now gives
\begin{equation}\label{eq:tr_geqslant_sum_integral}
\trace \big(S_{\gamma_p(r)}\big) \geqslant \sum_{j=1}^m \int_{r_0}^{r} \|\Ndr Y_j\|_g^2 +\kappa^2\|Y_j\|_g^2.
\end{equation}

We now compare the right-hand side of \eqref{eq:tr_geqslant_sum_integral} with a similar situation in the space form of curvature $-\kappa^2$.
Let $(\bar{M},\bar{g})= \R H^{m+1}(-\kappa^2)$ be that space form, $\bar{\nabla}$ its Levi-Civita connection, $\bar{N}= \R H^m(-\kappa^2) \hookrightarrow M$ be isometrically embedded, for instance as the equatorial hyperplane in the ball model, with a unit normal $\bar{\nu}$, and shape operator $\bar{S}$.
Let $\{\bar{\delr},\bar{E}_1,\ldots,\bar{E}_m\}$ be an orthonormal frame along a radial geodesic ${\bar{\gamma}}_{\bar{p}}$ obtained using the parallel transport of an orthonormal basis $\{\bar{\nu},\bar{e}_1,\ldots,\bar{e}_m\}$ of $T_{\bar{p}}\bar{M}$.
For $j\in \{1,\ldots,m\}$, let $\bar{Y}_j$ be the unique Jacobi field along ${\bar{\gamma}}_{\bar{p}}$ such that $\bar{\nabla}_{\bar{\delr}}\bar{Y}_j = \bar{S}\bar{Y}_j$ and $\bar{Y}_j(r-r_0) = \bar{E}_j(r-r_0)$.
Finally, let $X_j$ be defined by
\begin{equation*}
\forall t \in [0,r-r_0],\quad X_j(t) = \sum_{k=1}^m g\left(Y_j(r_0+t),E_k(r_0+t)\right)\bar{E}_k(t),
\end{equation*}
so that the following equalities hold
\begin{equation}\label{eq:norm_X_j_Y_j}
\begin{split}
\forall t \in [0,r-r_0],\quad \|X_j(t)\|_{\bar{g}} &= \|Y_j(r_0+t)\|_g,\\ \|\bar{\nabla}_{\bar{\delr}}X_j(t)\|_{\bar{g}} &= \|\Ndr Y_j(r_0+t)\|_g.
\end{split}
\end{equation}
Expanding the inequality $\|\bar{\nabla}_{\bar{\delr}}(\bar{Y}_j-X_j)\|_{\bar{g}}^2 + \kappa^2 \|\bar{Y}_j - X_j\|_{\bar{g}} \geqslant 0$, and integrating the developed expression on $[0,r-r_0]$ yields
\begin{equation}\label{eq:developed}
\begin{split}
\int_0^{r-r_0} \|\bar{\nabla}_{\bar{\delr}}X_j\|_{\bar{g}}^2 + \kappa^2\|X_j\|_{\bar{g}}^2 & \geqslant 2 \int_0^{r-r_0}\bar{g}(\bar{\nabla}_{\bar{\delr}}\bar{Y}_j,\bar{\nabla}_{\bar{\delr}}X_j) + \kappa^2 \bar{g}(\bar{Y}_j,{X}_j) \\
& \quad - \int_0^{r-r_0}\|\bar{\nabla}_{\bar{\delr}}\bar{Y}_j\|_{\bar{g}}^2 + \kappa^2 \|\bar{Y}_j\|_{\bar{g}}^2.
\end{split}
\end{equation}
Since $\bar{Y}_j$ is a Jacobi field, the integrands of the right-hand side satisfy
\begin{equation*}
\begin{split}
\bar{g}(\bar{\nabla}_{\bar{\delr}}\bar{Y}_j,\bar{\nabla}_{\bar{\delr}}X_j) + \kappa^2 \bar{g}(\bar{Y}_j,{X}_j) &= \frac{\dx}{\dx t} \bar{g}(\bar{\nabla}_{{\bar{\Ndr}}}\bar{Y}_j,X_j), \text{ and}\\
\|\bar{\nabla}_{\bar{\delr}}\bar{Y}_j\|_{\bar{g}}^2 + \kappa^2 \|\bar{Y}_j\|_{\bar{g}} ^2 & = \frac{\dx}{\dx t}\bar{g}(\bar{\nabla}_{\bar{\delr}}\bar{Y}_j,\bar{Y}_j).
\end{split}
\end{equation*}
Recall that $\bar{\nabla}_{\bar{\delr}}\bar{Y}_j = \bar{S}\bar{Y}_j$.
Since $N$ is isometrically embedded, it holds that $\bar{S}_{\bar{p}}=0$.
Also, recall that at $r-r_0$, $\bar{Y}_j(r-r_0) = X_j(r-r_0)$.
Hence, \eqref{eq:developed} yields
\begin{equation}
\label{eq:integral_geqslant_bar_S}
\int_0^{r-r_0} \|\bar{\nabla}_{\bar{\delr}}X_j\|_{\bar{g}}^2 + \kappa^2\|X_j\|_{\bar{g}}^2  \geqslant \bar{g}(\bar{S}\bar{E}_j(r-r_0),\bar{E}_j(r-r_0)).
\end{equation}
It now follows from equations \eqref{eq:tr_geqslant_sum_integral}, \eqref{eq:norm_X_j_Y_j} and \eqref{eq:integral_geqslant_bar_S} that we have the inequality $\trace\big(S_{\gamma_p(r)}\big) \geqslant \trace \big( \bar{S}_{\bar{\gamma}_{\bar{p}}(r-r_0)}\big)$.
Note that $\Sigma_{{\bar{\gamma}}_{\bar{p}}(t)} = \kappa\tanh(\kappa t) \Id_{{\{{\bar{\gamma}}_{\bar{p}}(t)\}}^{\perp}}$ satisfies the Riccati equation
\begin{equation*}
\bar{\nabla}_{\bar{\delr}} \Sigma =- \Sigma^2 - \bar{R}_{\bar{\delr}},
\end{equation*}
with initial data $\Sigma_{\bar{p}}=0$, so that $\bar{S}=\Sigma$.
This concludes the proof.
\end{proof}

\label{dfn:lambda}
We now return to the setting of this paper and consider a K\"ahler manifold $(M,g,J)$ with an essential subset $K$.
It is canonically oriented, and so is the co-oriented hypersurface $\dK$.
Let  $v_g$ and $v_{g|_{\dK}}$ be the Riemannian volume forms induced by the metrics $g$ and $g|_{\dK}$.
Since $E$ is a diffeomorphism, $E^*v_g$ and $\dx r \wedge v_{g|_{\dK}}$ are two volume forms on $\R_+\times \dK$, and they are proportional.
The \textit{volume density function} $\lambda$ is the unique positive function $\lambda \colon \R_+\times \dK \to \R$ such that $E^*v_g = \lambda \dx r \wedge v_{g|_{\dK}}$.
It is clear that $\lambda|_{\{0\}\times \dK}=1$.
The volume density function and the shape operator are related by the following differential equation
\begin{equation}
\label{eq:volume_density_trace}
\frac{\partial \lambda (r,p) }{\partial r} = \lambda(r,p) \,\trace (S_{\gamma_p(r)}).
\end{equation}
We shall now derive from Lemma \ref{lem:trace_bound_below} a lower bound on the volume density function.

\begin{proposition}\label{prop:volume_lower_bound}
Let $(M^{2n+2},g,J)$ be an \ALCH manifold of order $a>0$ with an essential subset $K$, such that the sectional curvature of $\bar{M\setminus K}$ is negative.
Let $\varepsilon \in (0,n + \frac{1}{2})$ be fixed.
Then there exist $r_0=r_0(\varepsilon) >0$ and $\Lambda_-= \Lambda_-(\varepsilon) >0$ such that the volume density function $\lambda$ satisfies
\begin{equation*}
\forall (r,p) \in [r_0,+\infty)\times \dK,\quad \lambda(r,p) \geqslant \Lambda_- e^{(n+\frac{1}{2}-\varepsilon)r}.
\end{equation*}
\end{proposition}
\begin{proof}
Let $\kappa = \frac{1}{2}- \frac{\varepsilon}{2n+1}>0$ and $r_0 = \max\left\{\frac{1}{a}\ln \big(\frac{\frac{1}{4}-\kappa^2}{C_0}\big),1\right\}>0$.
By definition of $r_0$,
\begin{equation*}
\forall r \geqslant r_0,\quad -\frac{1}{4} +C_0 e^{-ar} \leqslant -\kappa^2 <0.
\end{equation*}
According to Lemma \ref{lem:pinched_sectional_curvature}, the sectional curvature of linear planes based at points $\gamma_p(r)$, with $r\geqslant r_0$, have sectional curvature bounded above by $-\kappa^2$.
Hence, Lemma \ref{lem:trace_bound_below} yields the inequality
\begin{equation*}
\forall (r,p) \in [r_0,+\infty) \times \dK,\quad \trace \big( S_{\gamma_p(r)}\big) \geqslant (2n+1)\kappa \tanh(\kappa(r-r_0)).
\end{equation*}
The differential equation \eqref{eq:volume_density_trace} satisfied by $\lambda$ and $S$ now yields
\begin{equation*}
\forall (r,p) \in [r_0,+\infty) \times \dK,\quad \frac{\delr \lambda (r,p)}{\lambda (r,p)} \geqslant (2n+1)\kappa \tanh(\kappa(r-r_0)).
\end{equation*}
Integrating this last inequality yields
\begin{equation*}
\forall (r,p) \in [r_0,+\infty)\times \dK,\quad \lambda(r,p) \geqslant \lambda(r_0,p) \cosh^{2n+1}(\kappa(r-r_0)).
\end{equation*}
The result now follows by setting $\Lambda_- = \frac{e^{-(2n+1)\kappa r_0}}{2^{2n+1}} \min_{p\in \dK} \lambda(r_0,p)$, which exists and is positive since $\dK$ is compact and $\lambda$ positive and continuous, and by noticing that $(2n+1)\kappa = n+\frac{1}{2}-\varepsilon$.
\end{proof}


\section{Normal Jacobi fields estimates}
\label{section:normal_jacobi}
\noindent
In this section, we consider $(M^{2n+2},g,J)$ a fixed \ALCH manifold of order $a>0$ with an essential subset $K$, and we study the asymptotic behaviour of its normal Jacobi fields.
The geometric structure we wish to highlight being of contact nature, we choose not to work in coordinates.
Instead, we define some natural moving frames, which we call radially parallel orthonormal frames, in which the computations are convenient.

\subsection{Radially parallel orthonormal frame}

To the radial vector field $\delr$ is naturally associated the vector field $\Jdelr$.
Since the metric is K\"ahler, $\Jdelr$ is radially parallel.
It can be obtained using the parallel transport of $J\nu$ along the radial geodesics $(\gamma_p)_{p\in \dK}$.

\begin{definition}[Radially parallel orthonormal frame]
Let $\{J\nu,e_1,\ldots,e_{2n}\}$ be an orthonormal frame defined on an open subset $U\subset \dK$.
For $j\in \{1,\ldots,2n\}$, let $E_j$ be vector fields obtained by the parallel transport of $e_j$ along radial geodesics.
The orthonormal frame $\{\delr,\Jdelr,E_1,\ldots,E_{2n}\}$ on the cylinder $E(\R_+\times U)$ is called the \textit{radially parallel orthonormal frame} associated to $\{J\nu,e_1,\ldots,e_{2n}\}$.
\end{definition}

A radially parallel orthonormal frame is composed of radially parallel vector fields.
It is worth noting that $\{\delr,\Jdelr\}$ spans a complex line while $\{\delr,E_j\}$ spans a totally real plane if $j\in \{1,\ldots,2n\}$.
Moreover, $\{E_1,\ldots,E_{2n}\}$ is an orthonormal frame of the $J$-invariant subbundle $\{\delr,\Jdelr\}^{\perp}$ of $TM$.
\begin{lemma}\label{lem:sectional_curvature_radially_parallel_orthonormal_frame}
Let $\{\delr,\Jdelr,E_1,\ldots,E_{2n}\}$ be a radially parallel orthonormal frame.
Then the following holds
\begin{equation*}
\begin{cases}
\hfill|\sec(\delr,\Jdelr) + 1| \leqslant C_0e^{-ar}, \\
\hfill|R(\delr,\Jdelr,\delr,E_j)|  \leqslant C_0e^{-ar}, & \forall j\in \{1,\ldots,2n\},\\
\hfill|\sec(\delr,E_j) + \frac{1}{4}| \leqslant C_0e^{-ar}, & \forall j\in \{1,\ldots,2n\},\\
\hfill|R(\delr,E_i,\delr,E_j)|  \leqslant C_0e^{-ar}, & \forall i,j\in \{1,\ldots,2n\}, i\neq j.
\end{cases}
\end{equation*}
\end{lemma}
\begin{proof}
Let $X$ and $Y$ be unit vector fields.
The \ALCH condition gives
\begin{equation*}
|R(\delr,X,\delr,Y) - R^0(\delr,X,\delr,Y)| \leqslant C_0e^{-ar}.
\end{equation*}
The result follows from a direct computation of $R^0(\delr,X,\delr,Y)$ for the different values of $X$ and $Y$ in $\{\Jdelr, E_j\}$ and their orthogonal properties.
\end{proof}

\subsection{The Jacobi system}\label{section:jacobi_system}

Since $K$ is an essential subset, for $r\geqslant 0$,  $E(r,\cdot)$ is a diffeomorphism between $\dK$ and the level hypersurface at distance $r$ above $\dK$.

\begin{definition}[Local $1$-forms] \label{dfn:local-1-forms}
Let $\{\delr,\Jdelr,E_1,\ldots,E_{2n}\}$ be a radially parallel orthonormal frame on the cylinder $E(\R_+\times U)$.
We define on $U$ the family of 1-forms $\{\eta_r,\eta^1_r\ldots,\eta^{2n}_r\}$ by
\begin{equation*}
\begin{split}
\forall r \geqslant 0, \quad \eta_r &= e^{-r}E(r,\cdot)^* \left( g(\cdot, \Jdelr)|_{{\delr}^{\perp}}\right),\\
\forall j \in \{1,\ldots,2n\}, \forall r \geqslant 0,\quad  \eta^j_r &= e^{-\frac{r}{2}} E(r,\cdot)^* \left( g(\cdot, E_j)|_{{\delr}^{\perp}}\right).
\end{split}
\end{equation*}
\end{definition}

In other words, $\eta_r(v) = e^{-r}g(Y_v,\Jdelr)$ and $\eta_r^j(v) = e^{-\frac{r}{2}}g(Y_v,E_j)$.
Note that $\eta_r$ is defined on all of $\dK$ and does not depend on the choice of a radially parallel orthonormal frame.
With these notations, a normal Jacobi field $Y_v = E_*v$ along a radial geodesic $\gamma_p$ reads
\begin{equation}
\label{eq:component_jacobi}
Y_v = \eta_r(v) e^r \Jdelr + \sum_{j=1}^{2n} \eta_r^j(v) e^{\frac{r}{2}} E_j.
\end{equation}
On the cylinder $E(\R_+\times U)$, we also define the following $(2n+1)^2$ functions $\{u^i_k\}_{i,k\in \{0,\ldots,2n\}}$ by
\begin{equation}
\label{eq:def_uij}
u^i_k = - \begin{cases}
\sec(\delr,\Jdelr)+1 & \text{if } i=k=0,\\
e^{-\frac{r}{2}} R(\delr,\Jdelr,\delr,E_k) & \text{if } i = 0,k\in \{1,\ldots,2n\},\\
e^{\frac{r}{2}}R(\delr,E_i,\delr,\Jdelr) & \text{if } i\in \{1,\ldots,2n\}, k = 0,\\
R(\delr,E_i,\delr,E_k) & \text{if } i,k\in \{1,\ldots,2n\}, i\neq k,\\
\sec(\delr,E_i) + \frac{1}{4} & \text{if } i=k\in \{1,\ldots, 2n\}.
\end{cases}
\end{equation}

For the rest of this subsection, we shall fix the vector $v$ and henceforth drop out the variable $v$: we will write $\eta_r$ and $\eta^j_r$ instead of $\eta_r(v)$ and $\eta^j_r(v)$.
In addition, when a summation is involved, we set $\eta^0_r = \eta_r$ for the sake of compactness.
\begin{lemma}
The functions $\{\eta_r^i(v)\}_{i\in\{0,\ldots,2n\}}$ are solutions of the linear second order differential system
\begin{equation}
\label{eq:Jacobi_system}
\begin{cases}
\hfill \delr^2\eta_r + 2 \delr\eta_r   = \displaystyle\sum_{k=0}^{2n} u^0_k \eta^k_r,\\
\delr^2\eta^j_r + \delr\eta^j_r  = \displaystyle\sum_{k=0}^{2n} u^j_k \eta^k_r,& \forall j \in \{1,\ldots,2n\}.
\end{cases}
\end{equation}
\end{lemma}
\begin{proof}
Since the vectors of the radially parallel orthonormal frame are radially parallel, it follows that
\begin{equation*}
\Ndr(\Ndr Y_v) = (\delr^2\eta_r + 2 \delr\eta_r + \eta_r)e^r \Jdelr +\sum_{j=1}^{2n}\left(\delr^2{\eta^j_r} + \delr\eta^j_r + \frac{1}{4}\eta^j_r\right) e^{\frac{r}{2}}E_j.
\end{equation*}
Since $Y_v$ is a Jacobi field along $\gamma_p$, it holds that
\begin{equation*}
\Ndr(\Ndr Y_v) = -\eta_r e^r R(\delr,\Jdelr)\delr - \sum_{j=1}^{2n} \eta^j_re^{\frac{r}{2}} R(\delr,E_j)\delr.
\end{equation*}
The result then follows from an identification of the coefficients in the orthonormal frame $\{\delr,\Jdelr,E_1,\ldots,E_{2n}\}$.
\end{proof}
We now give a reformulation of the Jacobi system \eqref{eq:Jacobi_system} with integrals.
\begin{lemma}
The functions $\{\eta^i_r\}_{i\in\{0,\ldots,2n\}}$ are solutions of the integral system
\begin{equation}
\label{eq:integral_system}
\begin{cases}
\eta_r  &= \quad \displaystyle \eta_0 + \delr\eta_0\frac{1-e^{-2r}}{2} +\displaystyle\int_0^r\sum_{k=0}^{2n} u^0_k(\gamma_p(s)) \eta^k_ss \frac{1-e^{-2(r-s)}}{2}\dx s,\\
\eta^j_r  &=\quad \eta^j_0 + \delr\eta^j_0(1-e^{-r}) + \displaystyle \int_0^r\sum_{k=0}^{2n} u^j_k(\gamma_p(s)) \eta^k_s(1-e^{-(r-s)})\dx s,
\\ & \hfill  \forall j \in \{1,\ldots,2n\}.
\end{cases}
\end{equation}
\end{lemma}

\begin{proof}
We give details for the first equation, the other ones being proven similarly.
The differential equation satisfied by $\eta_r(v)$ (see equation \eqref{eq:Jacobi_system}) reads
\begin{equation*}
\delr\left(e^{2r}\delr\eta_r\right) = \sum_{k=0}^{2n} e^{2r}u^0_k \eta^k_r,
\end{equation*}
which integrates as
\begin{equation*}
\eta_r = \delr\eta_0 e^{-2r} + \int_0^r \sum_{k=0}^{2n} u^0_k(\gamma_p(s))\eta^k_s e^{-2(r-s)} \dx s.
\end{equation*}
A second integration now gives
\begin{equation*}
\eta_r = \eta_0 + \delr\eta_0\frac{1-e^{-2r}}{2} + \sum_{k=0}^{2n}\int_0^r \int_0^tu^0_k(\gamma_p(s))\eta^k_s e^{-2(t-s)}\dx s \dx t.
\end{equation*}
The function $(s,t) \in [0,r]^2 \mapsto u^0_k(\gamma_p(s))\eta^k_s e^{-2(t-s)}\mathbf{1}_{\{s\leqslant t\}}$ is measurable and bounded on a compact domain, and hence integrable.
It now follows from Fubini's Theorem that for $k\in \{0,\ldots,2n\}$, it holds that
\begin{equation*}
\begin{split}
\int_0^r \int_0^t u^0_k(\gamma_p(s))\eta^k_s e^{-2(t-s)} \dx s \dx t &= \int_{[0,r]^2} u^0_k(\gamma_p(s))\eta^k_s e^{-2(t-s)}\mathbf{1}_{\{s\leqslant t\}} \dx s \dx t \\
&= \int_{[0,r]^2}u^0_k(\gamma_p(s))\eta^k_s e^{-2(t-s)}\mathbf{1}_{\{s\leqslant t\}} \dx t \dx s \\
&= \int_0^ru^0_k(\gamma_p(s)) \eta^k_s e^{2s}\left(\int_s^r e^{-2t} \dx t\right) \dx s \\
&= \int_0^r u^0_k(\gamma_p(s)) \eta^k_s \frac{1-e^{-2(r-s)}}{2}\dx s.
\end{split}
\end{equation*}
The proof is now complete.
\end{proof}

Let $u \colon \R_+ \to \R_+$ be defined by $u(r) = \max_{i,k\in \{0,\ldots,2n\}}\left|u^i_k\left(\gamma_p(r)\right)\right|$.
\begin{lemma}\label{lem:sum_phi_j_upper_bound}
The following upper bound holds
\begin{equation*}
\forall r\geqslant 0,\quad \sum_{i=0}^{2n} |\eta^i_r| \leqslant \big((2n+1)\|S_p\|+n+1\big)\|v\| \exp\left((2n+1)\int_0^r u(s)\dx s\right).
\end{equation*}
\end{lemma}

\begin{proof}
First, note that for $0\leqslant r \leqslant s$, it holds that
\begin{equation*}
\max\left\{\left|\frac{1-e^{-2r}}{2}\right|, |1-e^{-r}|, \left|\frac{1-e^{-2(r-s)}}{2}\right|, \left|1-e^{-(r-s)}\right|\right\} \leqslant 1.
\end{equation*}
The triangle inequality applied to the integral system \eqref{eq:integral_system} thus yields
\begin{equation*}
\begin{cases}
|\eta_r| &\leqslant |\eta_0| + |\delr\eta_0| + \displaystyle\int_0^r u(s) \sum_{i=0}^{2n}|\eta^i_s| \dx s,\\
|\eta^j_r| &\leqslant |\eta^j_0| + |\delr\eta^j_0| +\displaystyle \int_0^r u(s) \sum_{i=0}^{2n}|\eta^i_s| \dx s,\quad \forall j\in \{1,\ldots,2n\}.
\end{cases}
\end{equation*}
Summing all these inequalities now yields the Gr\"onwall-like inequality
\begin{equation*}
\sum_{i=0}^{2n} |\eta^i_r| \leqslant \left(\sum_{i=0}^{2n}|\eta^i_0| + |\delr\eta^i_0|\right) + \int_0^r (2n+1)u(s) \sum_{i=0}^{2n}|\eta^i_s| \dx s.
\end{equation*}
Applying Gr\"onwall's inequality to $\sum_{i=0}^{2n}|\eta^i_r|$ thus shows that
\begin{equation*}
\sum_{i=0}^{2n} |\eta^i_r| \leqslant \left(\sum_{i=0}^{2n} |\eta^i_0| + |\delr\eta^i_0|\right) \exp\left( (2n+1)\int_0^r u(s) \dx s\right).
\end{equation*}
Recall that $\Jdelr|_{\dK} = J\nu$, and that $E_k(0) = e_k$ for $k\in \{1,\ldots,2n\}$.
Hence, $\eta_0 = g(v,J\nu)$, $\delr\eta_0 = g(S_pv-v,J\nu)$, $\eta^k_0 = g(v,e_k)$, and $\delr\eta^k_0 = g(Sv-\frac{1}{2}v,e_k)$ for $k\in \{1,\ldots,2n\}$, and the result directly follows from Cauchy-Schwarz inequality.
\end{proof}

We shall now show that if $M$ is \ALCH of order $a>\frac{1}{2}$, then the components $\{\eta_r^j(v)\}_{j\in \{0,\ldots,2n\}}$ converge as $r \to +\infty$ with a well understood decay.
\begin{proposition}\label{prop:convergence_phi_j}
Let $(M^{2n+2},g,J)$ be an \ALCH manifold of order $a>\frac{1}{2}$ with an essential subset $K$.
Let $\{\delr,\Jdelr,E_1,\ldots,E_{2n}\}$ be a radially parallel orthonormal frame on a cylinder $E(\R_+\times U)$, $p\in U$ and $v\in T_p\dK$.
Let $Y_v = E_*v$ be the normal Jacobi field along $\gamma_p$ associated to $v$.
If $\eta_r$ and $\{\eta_r^j\}_{j\in \{1,\ldots,2n\}}$ are the component functions of $Y_v$ defined by equation \eqref{eq:component_jacobi}, then there exists a constant $C>0$ depending only on $C_0$ and $a$, and constants $\eta_{\infty}$ and $\{\eta^j_{\infty}\}_{j\in\{1,\ldots,2n\}}$ such that
\begin{equation}
\label{eq:convergence_phi_j}
\begin{split}
\max\{|\delr\eta_r|,|\eta_r-\eta_{\infty}|\}  &\leqslant C\|v\|_g \begin{cases}
e^{-ar} & \text{if } \frac{1}{2}<a<\frac{3}{2},\\
(r+1)e^{-\frac{3}{2}r} & \text{if } a= \frac{3}{2},\\
e^{-\frac{3}{2}r} & \text{if } a > \frac{3}{2},
\end{cases} \\
\max\{|\delr\eta^j_r|,|\eta^j_r-\eta^j_{\infty}|\}  &\leqslant C\|v\|_g \begin{cases}
e^{-(a-\frac{1}{2})r} & \text{if } \frac{1}{2}<a<\frac{3}{2},\\
(r+1)e^{-r} & \text{if } a= \frac{3}{2},\\
e^{-r} & \text{if } a > \frac{3}{2},
\end{cases}\\ &  \hspace{3.9cm} \forall j \in \{1,\ldots,2n\}	.
\end{split}
\end{equation}
\end{proposition}

\begin{proof}
According to Lemma \ref{lem:sectional_curvature_radially_parallel_orthonormal_frame}, the definition of the functions $\{u^i_k\}_{i,k\in \{0,\ldots,2n\}}$ (see equation \eqref{eq:def_uij}) yields $0\leqslant u(r) \leqslant C_0 e^{-(a-\frac{1}{2})r}$ for all $r\geqslant 0$.
Therefore, $u$ is integrable on $\R_+$ and one has $\int_{\R_+} u \leqslant \frac{2C_0}{2a-1}$.
Since $\dK$ is compact, $\sup_{p\in \dK} \|S_p\|_g <+\infty$.
It follows from Lemma \ref{lem:sum_phi_j_upper_bound} that
\begin{equation}
\label{eq:uniform_bound_phi}
\forall r \geqslant 0,\quad \sum_{i=0}^{2n}|\eta^i_r| \leqslant c\|v\|_g,
\end{equation}
with $c= \big((2n+1)\sup_{p\in \dK}\|S_p\|_g + n +1\big) \exp\big(\frac{(4n+2)C_0}{2a-1}\big)$, which only depends on $a$ and $C_0$.
Putting this upper bound in the Jacobi system \eqref{eq:Jacobi_system} yields
\begin{equation*}
\label{eq:Jacobi_system_2}
\begin{cases}
\hfill|\delr^2\eta_r + 2 \delr\eta_r|   \leqslant \displaystyle c\|v\|_g\sum_{k=0}^{2n} |u^0_k|,\\
\hfill|\delr^2\eta^j_r + \delr\eta^j_r)|  \leqslant \displaystyle c\|v\|_g\sum_{k=0}^{2n} |u^j_k|, &\forall j \in \{1,\ldots,2n\}.
\end{cases}
\end{equation*}
Looking back at the definitions of the functions $\{u^i_k\}_{i,k\in\{0,\ldots,2n\}}$ (see equation \eqref{eq:def_uij}), Lemma \ref{lem:sectional_curvature_radially_parallel_orthonormal_frame} gives
\begin{equation}
\label{eq:decoupled_system}
\begin{cases}
\hfill		|\delr^2\eta_r + 2 \delr\eta_r|  &\leqslant \displaystyle c'\|v\|_ge^{-ar},\\
\hfill		|\delr^2\eta^j_r + \delr\eta^j_r|  &\leqslant \displaystyle c'\|v\|_ge^{-(a-\frac{1}{2})r},  \quad \forall j \in \{1,\ldots,2n\},
\end{cases}
\end{equation}
with $c' = (2n+1)C_0c$, which only depends on $C_0$ and $a$.
The result now follows from applying Lemma \ref{lem:equadiff_lemma} to each inequality of the decoupled system \eqref{eq:decoupled_system}, and recalling that we have upper bounds on the initial data $|\eta^i_0|$ and $|\delr\eta^i_0|$ that are linear in $\|v\|_g$.
Notice that a uniform constant $C>0$ in \eqref{eq:convergence_phi_j} can be obtained by taking the maximum of the $(2n+1)$ constants given by Lemma \ref{lem:equadiff_lemma}, which depends only on $a$ and $C_0$.
\end{proof}

Proposition \ref{prop:convergence_phi_j} yields a pointwise convergence of the coefficients of normal Jacobi fields with a sharp estimate on their asymptotic behaviour.
The following section focuses on their dependence with respect to the vector $v$.

\subsection{The local 1-forms}

We now consider again the dependence in $v$ of the coefficients $\eta_r(v)$ and $\{\eta^j_r(v)\}_{j\in \{1,\ldots,2n\}}$.
The limits given by Proposition \ref{prop:convergence_phi_j} are simply denoted by $\eta(v)$ and $\eta^j(v)$.
In these terms, Proposition \ref{prop:convergence_phi_j} becomes the following.

\begin{proposition}\label{prop:convergence_etar}
Let $(M^{2n+2},g,J)$ be an \ALCH manifold of order $a>\frac{1}{2}$, with an essential subset $K$.
Then there exists a continuous $1$-form $\eta$ on $\dK$ and a constant $C>0$ depending only on $C_0$ and $a$ such that $\forall r \in \R_+, \forall (p,v) \in T\dK$
\begin{equation}
\label{eq:estimate_eta}
\max\left\{\left|\delr \eta_r(v)\right|,|\eta_r(v)- \eta(v)|\right\}  \leqslant C\|v\|_g \begin{cases}
e^{-ar} & \text{if } \frac{1}{2}<a<\frac{3}{2},\\
(r+1)e^{-\frac{3}{2}r} & \text{if } a= \frac{3}{2},\\
e^{-\frac{3}{2}r} & \text{if } a > \frac{3}{2}.
\end{cases}
\end{equation}
Moreover, if $\{\delr,\Jdelr,E_1,\ldots,E_{2n}\}$ is a radially parallel orthonormal frame on the cylinder $E(\R_+\times U)$, for all $j \in \{1,\ldots, 2n\}$, there exists a continuous $1$-form $\eta^j$ on $U$ such that $\forall r \in \R_+, \forall (p,v) \in TU$
\begin{equation}
\label{eq:estimate_etai}
\max\left\{\left|\delr\eta_r^j(v)\right|,|\eta^j_r(v) - \eta^j(v)|\right\}  \leqslant C\|v\|_g \begin{cases}
e^{-(a-\frac{1}{2})r} & \text{if } \frac{1}{2}<a<\frac{3}{2},\\
(r+1)e^{-r} & \text{if } a= \frac{3}{2},\\
e^{-r} & \text{if } a > \frac{3}{2}.
\end{cases}
\end{equation}
\end{proposition}

\begin{proof}
Let $\eta(v) = \eta_{\infty}$ and $\eta^j(v) = \eta^j_{\infty}$ for $j\in \{1,\ldots,2n\}$ be defined pointwise, where $\eta_{\infty}$ and $\{\eta^j_{\infty}\}_{j\in \{1,\ldots,2n\}}$ are given by Proposition \ref{prop:convergence_phi_j}.
The claimed estimates \eqref{eq:estimate_eta} and \eqref{eq:estimate_etai} are obtained as a straightforward translation of Proposition \ref{prop:convergence_phi_j}.

In addition, \eqref{eq:estimate_eta} shows that $(\eta_r)_{r\geqslant 0}$ uniformly converges to $\eta$ on any compact subset of $T\dK$ as $r\to +\infty$.
Therefore, $\eta$ is continuous.
It is furthermore linear in the fibres as a pointwise limit of $1$-forms.
It follows that $\eta$ is a continuous $1$-form on $\dK$.
The exact same proof shows that if $j\in \{1,\ldots,2n\}$, then $(\eta^j_r)_{r\geqslant 0}$ uniformly converges to $\eta^j$ on every compact subset of $TU$ as $r\to +\infty$.
Therefore, $\eta^j$ is a continuous $1$-form on $U$.
The proof is now complete.
\end{proof}

\begin{corollary}\label{cor:uniform_bound_etai}
Under the assumptions of Proposition \ref{prop:convergence_etar}, there exists $c>0$ depending only on $a$ and $C_0$ such that
\begin{equation*}
\forall (p,v) \in T\dK,\quad |\eta(v)| \leqslant c\|v\|_g,
\end{equation*}
and
\begin{equation*}
\forall j \in \{1,\ldots,2n\}, \forall (p,v) \in TU,\quad |\eta^j(v)| \leqslant c \|v\|_g.
\end{equation*}
\end{corollary}

\begin{proof}
Equation \eqref{eq:uniform_bound_phi} gives the existence of $c$ depending only on $a$ and $C_0$ such that
\begin{equation*}
\forall r \geqslant 0, \forall (p,v) \in TU,\quad  |\eta_r(v)| + \sum_{j=1}^{2n} |\eta^j_r(v)| \leqslant c\|v\|_g.
\end{equation*}
The result follows by taking the limit as $r\to +\infty$.
Note that $c$ does not depend on $p$ or $U$, so that the upper bound is true on all of $\dK$ for $\eta$.
\end{proof}

\subsection{Normal Jacobi estimates}

Regarding Proposition \ref{prop:convergence_etar}, we henceforth assume $a>\frac{1}{2}$.
Consider a radially parallel orthonormal frame $\{\delr,\Jdelr,E_1,\ldots,E_{2n}\}$ on a cylinder $E(\R_+\times U)$, and let $\eta,\eta^1,\ldots,\eta^{2n}$ be the continuous $1$-forms given by Proposition \ref{prop:convergence_etar}.

\begin{definition}[Asymptotic vector fields]
\label{def:Z}
For $(r,p)\in \R_+\times U$, $v \in T_p\dK$, define the vectors $Z_v(\gamma_p(r)),Z_v'(\gamma_p(r))\in T_{\gamma_p(r)}M$ by
\begin{equation*}
\begin{split}
Z_v(\gamma_p(r)) &= \eta(v)e^r \Jdelr + \sum_{j=1}^{2n} \eta^j(v)e^{\frac{r}{2}}E_j,\\
Z_v'(\gamma_p(r)) &= \eta(v)e^r \Jdelr + \frac{1}{2}\sum_{j=1}^{2n} \eta^j(v)e^{\frac{r}{2}}E_j.
\end{split}
\end{equation*}
If $v$ is a vector field on $U$, we refer to $Z_v$ and $Z'_v$ as the \textit{asymptotic vector fields} related to the vector fields $Y_v$ and $SY_v$.
\end{definition}

\begin{proposition}\label{prop:Y_v-Z_v}
Assume that $a > \frac{1}{2}$.
Let $\{\delr,\Jdelr,E_1,\ldots,E_{2n}\}$ be a radially parallel orthonormal frame on a cylinder $E(\R_+\times U)$.
Then there exists $C>0$ depending only on $a$ and $C_0$ such that for any vector field $v$ on $U$,
\begin{equation*}
\max\{\|Y_v-Z_v\|_g,\|SY_v-Z'_v\|_g\} \leqslant C\|v\|_g \begin{cases}
e^{-(a-1)r} & \text{if } \frac{1}{2}<a<\frac{3}{2},\\
(r+1)e^{-\frac{r}{2}} & \text{if } a = \frac{3}{2},\\
e^{-\frac{r}{2}} &\text{if } a>\frac{3}{2}.
\end{cases}
\end{equation*}
\end{proposition}

\begin{proof}
By their very definition, it holds that
\begin{equation*}
Y_v - Z_v = (\eta_r(v)-\eta(v))e^r \Jdelr + \sum_{j=1}^{2n}(\eta_r^j(v)-\eta^j(v))e^{\frac{r}{2}}E_j.
\end{equation*}
Since $SY_v = \Ndr Y_v$, it also holds that
\begin{equation*}
SY_v - Z'_v =(\delr\eta_r(v)+\eta_r(v)-\eta(v))e^r\Jdelr + \sum_{j=1}^{2n}\big(\delr\eta^j_r(v)+\frac{1}{2}(\eta^j_r(v)-\eta^j(v))\big) e^{\frac{r}{2}}E_j.
\end{equation*}
The result follows from the triangle inequality and from Proposition \ref{prop:convergence_etar}.
\end{proof}

We conclude this section by stating an important upper bound on the growth of normal Jacobi fields.

\begin{lemma}\label{lem:important_estimates}
There exists a constant $c>0$ depending only on $a$ and $C_0$ such that for any local vector fields $v$ on $\dK$, it holds that
\begin{equation*}
\forall r\geqslant 0, \forall p,\quad \max\left\{\|Y_v(\gamma_p(r))\|_g,\|SY_v(\gamma_p(r))\|_g\right\} \leqslant c \|v\|_ge^r.
\end{equation*}
If moreover, $\eta|_p(v) = 0$, then
\begin{equation*}
\forall r\geqslant 0 ,\quad \max\{\|Y_v(\gamma_p(r))\|_g,\|SY_v(\gamma_p(r))\|_g\} \leqslant c \|v\|_ge^{\frac{r}{2}}.
\end{equation*}
\end{lemma}

\begin{proof}
By the triangle inequality, $\|Y_v\|_g \leqslant \|Z_v\|_g + \|Y_v-Z_v\|_g$, and the result follows from Corollary \ref{cor:uniform_bound_etai} and Proposition \ref{prop:Y_v-Z_v}.
The same proof applies for $SY_v$.
\end{proof}

\subsection{A volume growth upper bound}

In this section, we compute a pointwise upper bound on the volume density function $\lambda$ defined in section \ref{dfn:lambda} using the normal Jacobi fields estimates derived earlier.
The proof relies on an adapted choice for a basis of the tangent space $T_p\dK$ and on Hadamard's inequality on determinants.

\begin{proposition} \label{prop:volume_upper_bound}
Let $(M,g,J)$ be an \ALCH manifold of order $a>\frac{1}{2}$ with an essential subset $K$.
Let $\{\eta,\eta^1,\ldots,\eta^{2n}\}$ be the local continuous 1-forms associated to a radially parallel orthonormal frame $\{\delr,\Jdelr,E_1,\ldots,E_{2n}\}$ on a cylinder $E(\R_+\times U)$.
Let $p\in U$ and $k_p$ be the rank of the family $\{\eta|_p,\eta^1|_p\ldots,\eta^{2n}|_p\}$ as linear forms on $T_p\dK$.
\begin{itemize}
\item[•] If $\eta|_p=0$ then, there exists a constant $\Lambda_+=\Lambda_+(p) >0$ independent of $r$ such that
\begin{equation}
\label{eq:claimed_upper_bound}
\forall r \geqslant 0, \quad \lambda(r,p) \leqslant \Lambda_+ \begin{cases}
e^{(\frac{k_p}{2}-(a-1)(2n+1-k_p))r} & \text{if } \frac{1}{2} < a < \frac{3}{2} \\
(r+1)^{2n+1-k_p} e^{(k_p-n-\frac{1}{2})r} & \text{if } a = \frac{3}{2}, \\
e^{(k_p-n-\frac{1}{2})r} & \text{if } a> \frac{3}{2}.
\end{cases}
\end{equation}
\item[•] If $\eta|_p \neq 0$, then there exists a constant $\Lambda_+ = \Lambda_+(p)>0$ independent of $r$ such that
\begin{equation*}
\forall r \geqslant 0, \quad \lambda(r,p) \leqslant \Lambda_+ \begin{cases}
e^{(\frac{k_p+1}{2}-(a-1)(2n+1-k_p))r} & \text{if } \frac{1}{2} < a < \frac{3}{2} \\
(r+1)^{2n+1-k_p} e^{(k_p-n)r} & \text{if } a = \frac{3}{2}, \\
e^{(k_p-n)r} & \text{if } a> \frac{3}{2}.
\end{cases}
\end{equation*}
\end{itemize}
\end{proposition}

\begin{proof}
We first show the case where $\eta|_p = 0$.
Without loss of generality, one can assume that $\{\eta^1|_p,\ldots,\eta^{k_p}|_p\}$ generates the family $\{\eta|_p,\eta^1|_p,\ldots,\eta^{2n}|_p\}$.
Let $\{v_1,\ldots,v_{2n+1}\}$ be a basis of $T_p\dK$ such that $\eta^i|_p(v_j) = \delta^i_j$ for $(i,j) \in \{1,\ldots,k_p\}^2$, and  $v_{k_p+1},\ldots,v_{2n+1} \in \cap_{i=1}^{k_p} \ker \eta^i|_p$.
The volume density function is given by the relation
\begin{equation*}
\forall(r,p) \in \R_+\times \dK,\quad \lambda(r,p) = |\det \dx E(r,p)|,
\end{equation*}
where the determinant is taken in any orthonormal bases.
It follows that
\begin{equation*}
\forall r \geqslant 0,\quad \lambda(r,p) = \frac{|\det (\delr, Y_1,\ldots, Y_{2n+1})|}{|\det (\nu,v_1,\ldots,v_{2n+1})|},
\end{equation*}
where $Y_j = Y_{v_j}(\gamma_p(r))$, and all determinants are taken in orthonormal bases.
Hadamard's inequality on determinants now yields
\begin{equation*}
\forall r \geqslant 0,\quad \lambda(r,p) \leqslant \frac{1}{|\det(\nu,v_1,\ldots,v_{2n+1})|}\prod_{i=1}^{2n+1} \|Y_i\|_g.
\end{equation*}
According to Proposition \ref{prop:Y_v-Z_v}, there exists $C>0$ such that
\begin{equation*}
\forall i \in \{k_p+1,2n+1\},\quad \|Y_i\|_g \leqslant C \|v_i\|_g \begin{cases}
e^{-(a-1)r} & \text{if } \frac{1}{2} < a < \frac{3}{2}, \\
(r+1)e^{-\frac{r}{2}} & \text{if } a = \frac{3}{2}, \\
e^{-\frac{r}{2}} &\text{if } a >  \frac{3}{2},
\end{cases}
\end{equation*}
while Lemma \ref{lem:important_estimates} provides the existence of $c>0$ such that
\begin{equation*}
\forall i \in \{1,\ldots,k_p\},\quad \|Y_i\|_g \leqslant c\|v_i\|_ge^{\frac{r}{2}}.
\end{equation*}
The claimed upper bound \eqref{eq:claimed_upper_bound} now holds with $\Lambda_+ = C^{2n+1-k_p}c^{k_p}\frac{\prod_{i=1}^{2n+1}\|v_i\|_g}{|\det(v_1,\ldots,v_{2n+1})|}$.

In case $\eta|_p \neq 0$, the proof is similar.
Without loss of generality, one can assume that $\{\eta|_p,\eta^1|_p,\ldots,\eta^{k_p-1}|_p\}$ generates the whole family.
In that case, the first vector of the considered basis will have growth of order at most $e^r$, hence the extra $\frac{r}{2}$ term in the exponent of the upper bound.
\end{proof}

\subsection{The local coframe}

We shall now compare the two bounds on the volume growth given by Propositions \ref{prop:volume_lower_bound} and \ref{prop:volume_upper_bound}.
Note that these two bounds have been derived from considerations on the curvature of different nature.

\begin{proposition}\label{prop:coframe}
Let $(M^{2n+2},g,J)$ be an \ALCH manifold of order $a>1$, with an essential subset $K\subset M$, such that the sectional curvature of $\bar{M\setminus K}$ is non-positive.
Then the continuous $1$-form $\eta$ is non-vanishing.
Furthermore, if $\{\delr,\Jdelr,E_1,\ldots,E_{2n}\}$ is a radially parallel orthonormal frame on a cylinder $E(\R_+\times U)$, then $\{\eta,\eta^1,\ldots,\eta^{2n}\}$ is a continuous coframe on $U$.
\end{proposition}

\begin{proof}
Let $p\in \dK$ and let $\{\delr,\Jdelr,E_1,\ldots,E_{2n}\}$ be a radially parallel orthonormal frame on a cylinder $E(\R_+\times U)$, with $p\in U \subset \dK$.
Let $\{\eta,\eta^1,\ldots,\eta^{2n}\}$ be the continuous local 1-forms on $U$ given by Proposition \ref{prop:convergence_etar}, and $k_p$ be the rank of this family as linear forms on $T_p\dK$.
Let $\varepsilon \in (0,\min\{a-1,\frac{1}{2}\})$.
Propositions \ref{prop:volume_lower_bound} and \ref{prop:volume_upper_bound} yield the existence of $r_0=r_0(\varepsilon,a,C_0)>0$, $\Lambda_-=\Lambda_-(\varepsilon,a,C_0)>0$ and $\Lambda_+=\Lambda_+(a,C_0,p)>0$ such that
\begin{itemize}
\item If $\eta|_p = 0$, then for $r\geqslant r_0$, we get
\begin{equation}
\label{eq:etap=0}
\Lambda_- e^{(n+\frac{1}{2}-\varepsilon)r}\leqslant \lambda(r,p) \leqslant \Lambda_+ \begin{cases}
e^{(\frac{k_p}{2}-(a-1)(2n+1-k_p))r} & \text{if } 1 < a < \frac{3}{2}, \\
(r+1)^{2n+1-k_p} e^{(k_p-n-\frac{1}{2})r} & \text{if } a = \frac{3}{2}, \\
e^{(k_p-n-\frac{1}{2})r} & \text{if } a> \frac{3}{2}.
\end{cases}
\end{equation}
\item If $\eta|_p \neq 0$, then for $r\geqslant r_0$, we get
\begin{equation}
\label{eq:etapneq0}
\Lambda_- e^{(n+\frac{1}{2}-\varepsilon)r}\leqslant \lambda(r,p) \leqslant\Lambda_+ \begin{cases}
e^{(\frac{k_p+1}{2}-(a-1)(2n+1-k_p))r} & \text{if } 1 < a < \frac{3}{2}, \\
(r+1)^{2n+1-k_p} e^{(k_p-n)r} & \text{if } a = \frac{3}{2}, \\
e^{(k_p-n)r} & \text{if } a> \frac{3}{2}.
\end{cases}
\end{equation}
\end{itemize}
By contradiction, assume that $k_p < 2n+1$.
Then if $\eta|_p = 0$, a straightforward asymptotic comparison of the lower and upper bounds of \eqref{eq:etap=0} yields
\begin{equation*}
n+\frac{1}{2}- \varepsilon \leqslant \begin{cases}
n - (a-1) & \text{if } 1<a<\frac{3}{2}, \\
n - \frac{1}{2} & \text{if } a \geqslant \frac{3}{2},
\end{cases}
\end{equation*}
whereas if $\eta|_p \neq 0$, the same study in \eqref{eq:etapneq0} yields
\begin{equation*}
n+\frac{1}{2} - \varepsilon \leqslant \begin{cases}
n +\frac{1}{2}- (a-1) & \text{if } 1<a<\frac{3}{2} \\
n  & \text{if } a \geqslant \frac{3}{2}.
\end{cases}
\end{equation*}
Since $\varepsilon \in (0,\min\{a-1,\frac{1}{2}\})$, all cases lead to a contradiction.
It follows that $k_p = 2n+1$.
Thus, $\{\eta|_p,\eta^1|_p,\ldots,\eta^{2n}|_p\}$ is a linearly independent family.
This being true for all $p \in U$, $\{\eta,\eta^1,\ldots,\eta^{2n}\}$ is a coframe, and in particular, $\eta$ does not vanish on $U$.
The result follows.
\end{proof}

\begin{remark}
Notice that our technique would not have allowed us to draw a conclusion in the limit case $a=1$ when $\eta|_p\neq 0$.
\end{remark}

\begin{corollary}\label{cor:norm_Yu_greater}
Under the assumptions of Proposition \ref{prop:coframe}, there exists a constant $c>0$ such that
\begin{equation*}
\forall (p,v)\in T\dK, \forall r \geqslant 0,\quad \|Y_v\|_g \geqslant c\|v\|_ge^{\frac{r}{2}}.
\end{equation*}
\end{corollary}

\begin{proof}
For $r\geqslant 0$, we define a field of quadratic forms $Q_r$ on $\dK$ by the relation
\begin{equation*}
Q_r(v) = e^{-2r}g(Y_v,\Jdelr)^2 + e^{-r}\|Y_v^{\perp}\|_g^2,
\end{equation*}
where $Y_v^{\perp}$ denotes the orthogonal projection of $Y_v$ onto $\{\delr,\Jdelr\}^{\perp}$.
In a radially parallel orthonormal frame, $Q_r$ reads $Q_r = {\eta_r}^2 + \sum_{j=1}^{2n}{(\eta^j_r)}^2$.
It follows from Proposition \ref{prop:convergence_etar} that $Q_r$ locally uniformly converges on $\dK$, to a field of quadratic forms $Q_{\infty}$.
Since $\dK$ is compact, the convergence is uniform.
The limit locally reads $Q_{\infty} = \eta\otimes \eta + \sum_{j=1}^{2n} \eta^j\otimes \eta^j$, which is then positive definite by Proposition \ref{prop:coframe}.
Let $U\dK\subset T\dK$ be the unit sphere bundle.
Then $(r,u)\in [0,+\infty]\times U\dK \mapsto Q_r(u)$ is continuous on the compact set $[0,+\infty]\times U\dK$ and achieves its minimum at a point $(r_{\min},u_{\min})$.
Since $Q_{r_{\min}}$ is positive definite, this minimum is positive.
It follows by homogeneity that for any tangent vector $v$ on $\dK$, it holds
\begin{equation*}
\|Y_v\|_g^2 \geqslant e^r Q_r(v) \geqslant e^r \|v\|_g^2Q_{r_{\min}}(u_{\min}).
\end{equation*}
The result then follows.
\end{proof}

In particular, if $(g_r)_{r\geqslant 0}$ denotes the family of Riemannian metrics induced by the normal exponential chart on $\dK$ and if $(M,g,J)$ is \ALCH of order $a>1$, then there exists $c,C>0$ such that $ce^r g_0 \leqslant g_r \leqslant Ce^{2r}g_0$.

\subsection{Asymptotic development of the metric}\label{section:thmA}

We are now able to state again our first Theorem.

\begin{thmx}\label{thm:continuous}
Let $(M^{2n+2},g,J)$ be an \ALCH manifold of order $a>1$ with an essential subset $K$, such that the sectional curvature of $\bar{M\setminus K}$ is negative.
Then there exist a continuous non-vanishing $1$-form $\eta$ and a continuous Carnot-Carathéodory metric $\gamma_H$ on $\dK$, such that
\begin{equation*}
E^*g = \dx r^2 + e^{2r} \eta\otimes \eta + e^r \gamma_H+ h_r,
\end{equation*}
where $(h_r)_{r\geqslant 0}$ is a smooth family of symmetric bilinear forms on $\dK$ such that there exists $C>0$ with
\begin{equation*}
\forall r \geqslant 0, \forall u,v,\quad |h_r(u,v)| \leqslant C\|u\|_g\|v\|_g \begin{cases}
e^{(2-a)r} & \text{if } 1<a<\frac{3}{2}, \\
(r+1)e^{\frac{r}{2}} & \text{if } a = \frac{3}{2},\\
e^{\frac{r}{2}} & \text{if } a > \frac{3}{2}.
\end{cases}
\end{equation*}
Moreover, $\gamma_H$ is a continuous positive-definite metric on $H=\ker \eta$.
\end{thmx}

\begin{proof}
Let us first notice that Gauss Lemma yields the existence of a smooth family of metrics $(g_r)_{r\geqslant 0}$ on $\dK$ such that in the normal exponential chart, the metric reads $E^*g = \dx r ^2 + g_r$.
Let us fix $p  \in \dK$.
Let $U\subset \dK$ be an open neighbourhood of $p$ on which is defined the local coframe $\{\eta,\eta^1,\ldots,\eta^{2n}\}$ given by Proposition \ref{prop:coframe}.
Let $u$ and $v$ be two tangent vectors at $p$.

First, notice that $g_r(u,v) = g(Y_u,Y_v)$, where $Y_u$ and $Y_v$ are defined in subsection \ref{subsec:notations}.
Second, notice that
\begin{equation*}
g(Z_u,Z_v) = e^{2r}\eta(u)\eta(v)+e^r\sum_{j=1}^{2n}\eta^j(u)\eta^j(v),
\end{equation*}
where $Z_u$ and $Z_v$ are defined in Definition \ref{def:Z}.
Define $h_r$ by the relation
\begin{equation*}
h_r(u,v) = g(Z_u,Y_v-Z_v) + g(Y_u-Z_u,Z_v)+g(Y_u-Z_u,Y_v-Z_v),
\end{equation*}
It follows from the triangle inequality, Cauchy-Schwarz inequality, Proposition \ref{prop:Y_v-Z_v} and Lemma \ref{lem:important_estimates}, that there exists a constant $C>0$ independent of $r$, $p$, $u$ and $v$ such that
\begin{equation}
\label{eq:h_r}
|h_r(u,v)| \leqslant C\|u\|_g\|v\|_g \begin{cases}
e^{(2-a)r} & \text{if } 1<a<\frac{3}{2}, \\
(r+1)e^{\frac{r}{2}} & \text{if } a = \frac{3}{2},\\
e^{\frac{r}{2}} & \text{if } a > \frac{3}{2}.
\end{cases}
\end{equation}
Let $\gamma_H$ be defined by the relation $\gamma_H = \lim_{r\to +\infty} e^{-r}\left( g_r - e^{2r}\eta\otimes \eta \right)$.
Notice that
\begin{equation*}
g_r = e^{2r}\eta\otimes \eta + e^r \sum_{j=1}^{2n} \eta^j\otimes \eta^j + h_r,
\end{equation*}
so that it follows from \eqref{eq:h_r} that $\gamma_H$ locally reads $\gamma_H = \sum_{j=1}^{2n} \eta^j\otimes \eta^j$, and $\gamma_H$ is hence continuous and positive semi-definite.
Since $\{\eta,\eta^1,\ldots,\eta^{2n}\}$ is a local coframe, $\gamma_H$ is thus non-degenerate on $H = \ker \eta$.
This concludes the proof.
\end{proof}

\begin{corollary}
Under the assumptions of Theorem \ref{thm:continuous}, there exists a unique continuous vector field $\xi$ on $\dK$ such that
\begin{enumerate}
\item $\eta(\xi) = 1$,
\item $\gamma_H(\xi,\cdot) = 0$.
\end{enumerate}
\end{corollary}
\begin{proof}
Let us locally define the vector field $\xi$ as the first vector of the dual basis of the local coframe $\{\eta,\eta^1,\ldots,\eta^{2n}\}$.
Then $\xi$ is uniquely characterized by 1. and 2., so that $\xi$  does not depend on the choice of the radially parallel orthonormal frame defining the coframe, and is defined on all of $\dK$.
Since the local coframe is continuous, so is $\xi$.
This concludes the proof.
\end{proof}

\begin{definition}[Canonical elements at infinity]
Under the assumptions of Theorem \ref{thm:continuous}, we call $\eta$ the \textit{canonical $1$-form at infinity}, $H = \ker \eta$ the \textit{canonical distribution at infinity}, $\gamma_H$ the \textit{canonical Carnot-Carathéodory metric at infinity} and $\xi$ the \textit{canonical vector field at infinity}.
\end{definition}
We call these elements canonical because they are global and do not depend on the choice of any radially parallel orthonormal frame.
The analogy with the model case is obvious: $\eta$ corresponds to the contact form $\theta$, $\gamma_H$ corresponds to the Levi form $\gamma$ and $\xi$ corresponds to the Reeb vector field of $\theta$.
Notice that no assumption on $\nabla R$ has been made yet.
In the following section, we shall pursue the study of this analogy, and show that under stronger assumptions on $R$ and $\nabla R$, the canonical elements at infinity have $\mathcal{C}^1$ regularity, and that $\eta$ is indeed a contact form.


\section{The contact structure}
\label{section:contact_structure}
Let $(M,g,J)$ be an \ALCH manifold of order $a>1$, with an essential subset $K$, such that the sectional curvature of $\bar{M\setminus K}$ is non-positive.
We aim to show that the canonical $1$-form $\eta$ is a contact form of class $\mathcal{C}^1$.
To do so, we show that the family of $1$-forms $\{\eta_r\}_{r\geqslant 0}$ defined in Definition \ref{dfn:local-1-forms} converges to $\eta$ in $\mathcal{C}^1$ topology.

Let us first show the following Lemma.

\begin{lemma}\label{lem:L_uetav}
If $u$ and $v$ are vector fields on $\dK$, then
\begin{equation*}
\forall r \geqslant 0, \quad
(\mathcal{L}_u\eta_r)(v) = e^{-r}\big( g(\nabla_{Y_v}Y_u,\Jdelr) + g(Y_v,\nabla_{Y_u}\Jdelr)\big).
\end{equation*}
\end{lemma}
\begin{proof}
By basic properties of the Lie derivative, it holds that
\begin{equation*}
(\mathcal{L}_u\eta_r)(v) = u \cdot \eta_r(v) - \eta_r([u,v]).
\end{equation*}
Hence, the equalities $Y_u=E_*u$ and $Y_v = E_*v$ together with the naturality of the Lie bracket yield
\begin{equation*}
(\mathcal{L}_u\eta_r)(v) = Y_u \cdot e^{-r}g(Y_v,\Jdelr) - e^r g([Y_u,Y_v],\Jdelr).
\end{equation*}
Recall that $Y_u$ and $\delr$ are orthogonal.
Since $\delr$ is the gradient of the distance function $r$, it follows that $Y_u\cdot e^rg(Y_v,\Jdelr) = e^r Y_u\cdot g(Y_v,\Jdelr)$.
The torsion-free property of the Levi-Civita connection concludes the proof.
\end{proof}

Computations show that in order to estimate the asymptotic behaviour of the Lie derivative $(\mathcal{L}_u\eta_r)(v)$, an upper-bound on the norm of $\nabla_{Y_v}Y_u$ is needed.
This upper bound in itself is obtained by controlling the norm of $(\nabla_{Y_v}S)Y_u$.
The following section takes on the tedious task of providing such upper bounds.

\subsection{Order one estimates}

In this section, $u$ and $v$ are fixed vector fields on $\dK$.

\begin{lemma}\label{lem:F^2'}
The following holds
\begin{equation*}
\begin{split}
\frac{1}{2}\delr\|(\nabla_{Y_v}S)Y_u\|_g^2 &= R(\delr,Y_v,SY_u,S(\nabla_{Y_v}S)Y_u) - R(\delr,SY_v,Y_u,(\nabla_{Y_v}S)Y_u)\\
& \quad - R(SY_v,Y_u,\delr,(\nabla_{Y_v}S)Y_u) - R(\delr,Y_u,SY_v,(\nabla_{Y_v}S)Y_u) \\
& \quad - (\nabla_{Y_v}R)(\delr,Y_u,\delr,(\nabla_{Y_v}S)Y_u) - g(S(\nabla_{Y_v}S)Y_u,(\nabla_{Y_v}S)Y_u).
\end{split}
\end{equation*}
\end{lemma}

\begin{proof}
The extension of the covariant derivative to the whole tensor algebra gives the equality $(\nabla_{Y_v}S)Y_u = \nabla_{Y_v}(SY_u) - S \nabla_{Y_v}Y_u$.
It then follows that
\begin{equation}
\label{eq:add}
\Ndr ((\nabla_{Y_v}S)Y_u)) = \Ndr\big( \nabla_{Y_v}(SY_u)\big) - (\Ndr S)\nabla_{Y_v}Y_u - S\Ndr(\nabla_{Y_v}Y_u).
\end{equation}
From Lemma \ref{lem:Y_v}, $[\delr,Y_v]=0$, $\Ndr Y_u = SY_u$ and $\Ndr (SY_u) = -R_{\delr}Y_u$.
It follows from our convention on the Riemann curvature tensor (see \eqref{eq:convention_curv}) that
\begin{equation}
\label{eq:5.1.2}
\begin{split}
\Ndr\big( \nabla_{Y_v}(SY_u)\big) &= -R(\delr,Y_v)SY_u + \nabla_{Y_v}(-R_{\delr}Y_u) \\
&= -R(\delr,Y_v)SY_u - (\nabla_{Y_v}R)(\delr,Y_u)\delr \\
&\quad - R(SY_v,Y_u)\delr - R(\delr,\nabla_{Y_v}Y_u)\delr\\
&\quad - R(\delr,Y_u)SY_v,
\end{split}
\end{equation}
as well as
\begin{equation}
\label{eq:subs1}
\begin{split}
S\Ndr (\nabla_{Y_v}Y_u) &= -SR(\delr,Y_v)Y_u + S\nabla_{Y_v}(SY_u)\\
&= -SR(\delr,Y_v)Y_u + S^2\nabla_{Y_v}Y_u + S\big(\nabla_{Y_v}S)Y_u\big).
\end{split}
\end{equation}
The Riccati equation \eqref{eq:riccati} for $S$ yields
\begin{equation}\label{eq:subs2}
(\Ndr S) \nabla_{Y_v}Y_u = -S^2 \nabla_{Y_y}Y_u - R(\delr,\nabla_{Y_v}Y_u)\delr.
\end{equation}
Inserting \eqref{eq:5.1.2}, \eqref{eq:subs1} and \eqref{eq:subs2} into \eqref{eq:add} now gives
\begin{equation*}
\begin{split}
\Ndr ((\nabla_{Y_v}S)Y_u)) &= SR(\delr,Y_v)Y_u - R(\delr,Y_v)SY_u\\
&\quad - R(SY_v,Y_u)\delr - R(\delr,Y_u)SY_v\\
&\quad - (\nabla_{Y_v}R)(\delr,Y_u)\delr - S(\nabla_{Y_v}S)Y_u.
\end{split}
\end{equation*}
The result now follows from the symmetry of $S$ and the equality
\begin{equation*}
\frac{1}{2}\delr\|(\nabla_{Y_v}S)Y_u\|_g^2 = g(\Ndr ((\nabla_{Y_v}S)Y_u)),(\nabla_{Y_v}S)Y_u).
\end{equation*}
\end{proof}

\begin{proposition}
\label{prop:bound_nablaS}
If in addition $(M,g,J)$ has \ALS property of order $b>0$, then there exists $c>0$ such that
\begin{equation*}
\|(\nabla_{Y_v}S)Y_u\|_g \leqslant c \|u\|_g\|v\|_ge^{2r}.
\end{equation*}
\end{proposition}

\begin{proof}
Let $p\in \dK$.
Define $F(r) = \|((\nabla_{Y_v}S)Y_u)({\gamma_p(r)})\|_g$.
Since the norm is Lipschitz and the entries are smooth, $F$ is locally Lipschitz.
It follows from Rademacher's Theorem that $F$ is almost everywhere differentiable, and Lemma \ref{lem:F^2'} yields the almost everywhere equality
\begin{equation*}
\begin{split}
F' F &= R(\delr,Y_v,SY_u,S(\nabla_{Y_v}S)Y_u) - R(\delr,SY_v,Y_u,(\nabla_{Y_v}S)Y_u)\\
& \quad - R(SY_v,Y_u,\delr,(\nabla_{Y_v}S)Y_u) - R(\delr,Y_u,SY_v,(\nabla_{Y_v}S)Y_u) \\
& \quad - (\nabla_{Y_v}R)(\delr,Y_u,\delr,(\nabla_{Y_v}S)Y_u) - g(S(\nabla_{Y_v}S)Y_u,(\nabla_{Y_v}S)Y_u).
\end{split}
\end{equation*}
Since $S$ is positive, $ g(S(\nabla_{Y_v}S)Y_u,(\nabla_{Y_v}S)Y_u) \geqslant 0$.
Therefore, it follows that
\begin{equation*}
F'F \leqslant  + (4\|R\|_g \|S\|_g + \|\nabla R\|_g) \|Y_v\|_g\|Y_u\|_gF \quad a.e.
\end{equation*}
Recall that $\|R\|_g$  (Lemma \ref{lem:R_bounded}), $\|\nabla R\|_g$ (\ALS assumption) and $\|S\|_g$ (Lemma \ref{prop:norm_S_uniformly_bounded}) are uniformly bounded.
In addition, recall that there exists a constant $c_1>0$ such that $\|Y_u\|_g\leqslant c_1\|u\|_ge^r$ and $\|Y_v\|_g\leqslant c_1\|v\|_ge^r$ (Lemma \ref{lem:important_estimates}).
Hence, there exists a constant $c>0$ such that
\begin{equation*}
F'F \leqslant c\|u\|_g\|v\|_ge^{2r} F \quad a.e,
\end{equation*}
and therefore,
\begin{equation*}
F' \leqslant c\|u\|_g\|v\|_g e^{2r} \quad a.e,
\end{equation*}
which is true even at points where $F$ vanishes since at those points, it achieves a minimum and its derivative thus vanishes.
The result now follows from a straightforward integration.
\end{proof}

\begin{remark}
This bound is sharp, since $g((\nabla_{Y_v}S)Y_u,\delr) = -g(SY_v,SY_u)$ is equivalent to $-\eta(u)\eta(v)e^{2r}$.
It is however worth noticing that it can be proven that the normal component $(\nabla_{Y_v}S)Y_u^{\perp}=(\nabla_{Y_v}S)Y_u - g((\nabla_{Y_v}S)Y_u,\delr)\delr$ can be bounded, if $\min\{a,b\}>\frac{1}{2}$, as follows
\begin{equation*}
\|(\nabla_{Y_v}S)Y_u^{\perp}\|_g \leqslant\tilde{c}\|u\|_g\|v\|_g e^{\frac{3}{2}r}.
\end{equation*}
The proof is more intricate, and relies on the study of the curvature terms.
For example, the terms in $R^0(\delr,Y_v,SY_u,S(\nabla_{Y_v}S)Y_u)$ are products of the form $g(X,\Jdelr)g(Y,JZ)$ with $X$, $Y$, $Z$ normal to $\delr$, which kills the $\delr$ direction.
\end{remark}

We shall now give a bound on the growth of $\nabla_{Y_v}Y_u$.
The proof is slightly different from that of Proposition \ref{prop:bound_nablaS} as it requires the use of Gr\"onwall's inequality.

\begin{proposition}\label{prop:bound_nablaY_vY_u}
Assume that $a>\frac{1}{2}$ and $b>0$.
Then there exists a constant $c>0$ such that
\begin{equation*}
\|(\nabla_{Y_v}Y_u)(\gamma_p(r))\|_g \leqslant c(\|(\nabla_{Y_v}Y_u)(p)\|_g+ \|u\|_g\|v\|_g) e^{2r}.
\end{equation*}
\end{proposition}

\begin{proof}
First, notice that $\Ndr (\nabla_{Y_v}Y_u) = -R(\delr,Y_v)Y_u + \nabla_{Y_v}(SY_u)$ and moreover that $\nabla_{Y_v}(SY_u) = (\nabla_{Y_v}S)Y_u + S\nabla_{Y_v}Y_u$.
Hence
\begin{equation}\label{eq:G'^2}\begin{split}
\frac{1}{2}\delr \|\nabla_{Y_v}Y_u\|_g^2 &= -R(\delr,Y_v,Y_u,\nabla_{Y_v}Y_u) + g((\nabla_{Y_v}S)Y_u,\nabla_{Y_v}Y_u) \\
&\quad  +g(S\nabla_{Y_v}Y_u,\nabla_{Y_v}Y_u).
\end{split}
\end{equation}
Let $p\in \dK$ be fixed and let $G(r) = \|(\nabla_{Y_v}Y_u)({\gamma_p(r)})\|_g$.
Then $G$ is almost everywhere differentiable and the triangle inequality together with Cauchy-Schwarz inequality applied to \eqref{eq:G'^2} yields the almost everywhere inequality
\begin{equation*}
G'G \leqslant \|R\|_g\|Y_u\|_g\|Y_v\|_gG + \|(\nabla_{Y_v}S)Y_u\|_gG + \|S\|_gG^2\quad a.e,
\end{equation*}
from which is deduced the following almost everywhere inequality
\begin{equation*}
G' \leqslant \|R\|_g\|Y_u\|_g\|Y_v\|_g + \|(\nabla_{Y_v}S)Y_u\|_g + \|S\|_gG \quad a.e,
\end{equation*}
which is true even at points where $G$ vanishes since at those points, it achieves a minimum and its derivative thus vanishes.
Estimates on the growth of normal Jacobi fields (Lemma \ref{lem:important_estimates}), a uniform bound on $\|R\|_g$ (Lemma \ref{lem:R_bounded}) together with Proposition \ref{prop:bound_nablaS} give the existence of a constant $C>0$ such that
\begin{equation*}
G' - G \leqslant C\|u\|_g\|v\|_ge^{2r} + (\|S\|_g-1)G \quad a.e.
\end{equation*}
Multiplying both sides by $e^{-r}$ and integrating gives the Gr\"onwall-like inequality
\begin{equation*}
\forall r \geqslant 0,\quad G(r) e^{-r} \leqslant G(0) +C\|u\|_g\|v\|_g (e^r-1) + \int_0^r (\|S\|_g-1)(G(s)e^{-s})\dx s.
\end{equation*}
Finally, recall from Proposition \ref{prop:norm_S_uniformly_bounded} that
\begin{equation*}
\|S\|_g - 1 \leqslant \epsilon_a(r)=C' \begin{cases}
e^{-ar} & \text{if } \frac{1}{2} < a < 2, \\
(r+1)e^{-2r} & \text{if } a = 2,\\
e^{-2r} & \text{if } a >2,
\end{cases}
\end{equation*}
for some constant $C'$ independent of $(r,p,u,v)$.
Hence, Gr\"onwall's inequality yields
\begin{equation*}
\forall r \geqslant 0,\quad G(r)e^{-r} \leqslant (G(0) + C\|u\|_g\|v\|_ge^r)\exp\left( \int_0^r \epsilon_a(s)\dx s\right).
\end{equation*}
It follows that there exists $c>0$ such that
\begin{equation*}
\|(\nabla_{Y_v}Y_u)(\gamma_p(r))\|_g \leqslant c(\|(\nabla_{Y_v}Y_u)(p)\|_g + \|u\|_g\|v\|_g)e^{2r}.
\end{equation*}
\end{proof}

\begin{remark}
\begin{itemize}
\item This bound is sharp since $g(\nabla_{Y_v}Y_u,\delr) = - g(SY_v,Y_u)$ is equivalent to $-\eta(u)\eta(v)e^{2r}$.
Similarly to Proposition \ref{prop:bound_nablaS}, it is possible to give a sharper bound on the normal component, of order $e^{\frac{3}{2}r}$.
\item Since $\nabla_{Y_v}Y_u$ is not tensorial in $u$, one cannot hope to find a constant $c>0$ independent of $u$ such that $\|\nabla_{Y_v}Y_u\|_g \leqslant c \|u\|_g\|v\|_ge^{2r}$.
However, when $u$ and $v$ are fixed vector fields, the constant in front of the exponential term is continuous with respect to $p$, and hence uniformly bounded on all compact subsets on which it is defined.
\end{itemize}
\end{remark}

\subsection{The contact form}

From now on, we assume that $(M,g,J)$ is \ALCH and \ALS of orders $a>1$ and $b>0$, and also that the sectional curvature of $\bar{M\setminus K}$ is negative.
In this section, we prove that the canonical 1-form at infinity is of class $\mathcal{C}^1$ and is contact.
We first show the following computational Proposition.

\begin{proposition}\label{prop:convergence_to_alpha}
Assume that $\min\{a,b\}>1$.
Let $u$ and $v$ be local tangent vector fields on $\dK$.
Define
\begin{equation*}
f(r,p)  = g(\nabla_{Y_v}Y_u,\Jdelr) + g(\nabla_{Y_u}\Jdelr,Y_v).
\end{equation*}
Then there exists a constant $c>0$ independent of $(r,p,u,v)$ and a continuous function $\alpha(p)$ such that
\begin{equation*}
|f(r,p)-{e^r}\alpha(p) |\leqslant c(\|u\|_g\|v\|_g + \|\nabla_{Y_v}Y_u\|_g)\begin{cases}
e^{(2-\min\{a,b\})r} & \text{if } 1<\min\{a,b\} < 3,\\
(r+1)e^{-r} & \text{if } \min\{a,b\} ={3}, \\
e^{-r} & \text{if } \min\{a,b\}>3.
\end{cases}
\end{equation*}
\end{proposition}

\begin{proof}
Our convention on the curvature yields $\Ndr (\nabla_{Y_v}Y_u) = -R(\delr,Y_v)Y_u+ \nabla_{Y_v}(SY_u)$ and $\Ndr (\nabla_{Y_u}\Jdelr) = -R(\delr,Y_u)\Jdelr$.
Therefore,
\begin{equation*}
\begin{split}
\delr f &= -R(\delr,Y_v,Y_u,\Jdelr) + g(\nabla_{Y_v}(SY_u),\Jdelr)\\
&\quad -R(\delr,Y_u,\Jdelr,Y_v) + g(\nabla_{Y_u}\Jdelr,SY_v),
\end{split}
\end{equation*}
which turns out, by the Bianchi identity, to be equal to the following
\begin{equation}\label{eq:delrf}
\delr f = R(\delr,\Jdelr,Y_v,Y_u) + g(\nabla_{Y_v}(SY_u),\Jdelr) + g(\nabla_{Y_u}\Jdelr,SY_v).
\end{equation}
Notice that
\begin{equation*}\begin{split}
\Ndr (\nabla_{Y_v}(SY_u)) &= -R(\delr,Y_v)SY_u + \nabla_{Y_v}(\Ndr(SY_u))\\
&= -R(\delr,Y_v)SY_u + \nabla_{Y_v}(-R_{\delr}Y_u)\\
&= -R(\delr,Y_v)SY_u - (\nabla_{Y_v}R)(\delr,Y_u)\delr - R(SY_v,Y_u)\delr\\
&\quad - R(\delr,\nabla_{Y_v}Y_u)\delr - R(\delr,Y_u)SY_v.
\end{split}
\end{equation*}
It then follows that
\begin{equation}\label{eq:messed_up} \begin{split}
\delr \delr f &= (\Ndr R)(\delr,\Jdelr,Y_v,Y_u) + R(\delr,\Jdelr,SY_v,Y_u) \\
&\quad + R(\delr,\Jdelr,Y_v,SY_u) -R(\delr,Y_v,SY_u,\Jdelr)\\
&\quad  - (\nabla_{Y_v}R)(\delr,Y_u,\delr,\Jdelr) - R(SY_v,Y_u,\delr,\Jdelr)\\
&\quad -R(\delr,\nabla_{Y_v}Y_u,\delr,\Jdelr) -R(\delr,Y_u,SY_v,\Jdelr) \\
&\quad -R(\delr,Y_u,\Jdelr,SY_v) -R(\delr,Y_v,\delr,\nabla_{Y_u}\Jdelr).
\end{split}
\end{equation}
Note that $R(\delr,\Jdelr,SY_v,Y_u) - R(SY_v,Y_u,\delr,\Jdelr)=0$ because of the symmetry of $R$.
Similarly, notice that $-R(\delr,Y_u,SY_v,\Jdelr) - R(\delr,Y_u,\Jdelr,SY_v) = 0$.
Hence, equation \eqref{eq:messed_up} becomes

\begin{equation*}
\begin{split}
\delr\delr f &= (\Ndr R) (\delr,\Jdelr,Y_v,Y_u) - (\nabla_{Y_v} R) (\delr,Y_u,\delr,\Jdelr)\\
& \quad + R(\delr,\Jdelr,Y_v,SY_u) - R(\delr,Y_v,SY_u,\Jdelr)\\
& \quad -R(\delr,\nabla_{Y_v}Y_u,\delr,\Jdelr) - R(\delr,Y_v,\delr,\nabla_{Y_u}\Jdelr).
\end{split}
\end{equation*}
Let $\mathbf{k^0}$ be defined as
\begin{equation*}\begin{split}
\mathbf{k^0} &= R^0(\delr,\Jdelr,Y_v,SY_u) - R^0(\delr,Y_v,SY_u,\Jdelr)\\
& \quad -R^0(\delr,\nabla_{Y_v}Y_u,\delr,\Jdelr) - R^0(\delr,Y_v,\delr,\nabla_{Y_u}\Jdelr).
\end{split}
\end{equation*}
From Lemma \ref{appendix:thmB1} in the Appendix, it holds that
\begin{equation*}\begin{split}
\mathbf{k}^0 &= -\frac{1}{2}g(SY_u,JY_v) -\frac{1}{4}g(SY_u,JY_v) + g(\nabla_{Y_v}Y_u,\Jdelr) + \frac{1}{4}g(Y_v,\nabla_{Y_u}\Jdelr)\\
&= -\frac{3}{4}g(SY_u,JY_v) + \frac{1}{4}g(Y_v,\nabla_{Y_u}\Jdelr) + g(\nabla_{Y_v}Y_u,\Jdelr).
\end{split}
\end{equation*}
Note that $g(SY_u,JY_v) = g(\nabla_{Y_u}\delr,JY_v) = -g(J\nabla_{Y_u}\delr,Y_v)=-g(\nabla_{Y_u}\Jdelr,Y_v)$.
It surprisingly turns out that
\begin{equation*}
\mathbf{k}^0 = g(\nabla_{Y_v}Y_u,\Jdelr) + g(\nabla_{Y_u}\Jdelr,\Jdelr) = f.
\end{equation*}
Hence, $f$ a is solution to the second order linear differential equation
\begin{equation*}
\delr\delr f - f = h,
\end{equation*}
where $h$ is given by
\begin{equation}\label{eq:h}
\begin{split}
h&= (\Ndr R) (\delr,\Jdelr,Y_v,Y_u) - (\nabla_{Y_v} R) (\delr,Y_u,\delr,\Jdelr)\\
& \quad + (R-R^0)(\delr,\Jdelr,Y_v,SY_u) - (R-R^0)(\delr,Y_v,SY_u,\Jdelr)\\
& \quad -(R-R^0)(\delr,\nabla_{Y_v}Y_u,\delr,\Jdelr) - (R-R^0)(\delr,Y_v,\delr,\nabla_{Y_u}\Jdelr).
\end{split}
\end{equation}
By classical ODE theory, $f$ reads
\begin{equation*}
f = \frac{e^r}{2}\big({f(0)+\delr f(0)} +\int_0^re^{-s}h(s)\dx s\big) +\frac{e^{-r}}{2}\big(f(0) -\delr f(0) - \int_0^r e^s h(s) \dx s \big).
\end{equation*}
Since $\nabla_{Y_u}\Jdelr = JSY_u$, it holds that $\|JSY_u\|_g = \|SY_u\|_g \leqslant \|S\|_g\|Y_u\|_g$, and equation \eqref{eq:h} yields the following upper bound on $h$
\begin{equation*}
|h| \leqslant (3\|R-R^0\|_g\|S\|_g + 2\|\nabla R\|_g)\|Y_u\|_g\|Y_v\|_g + \|R-R^0\|_g\|\nabla_{Y_v}Y_u\|_g.
\end{equation*}
From the normal Jacobi estimates, the uniform bound on $\|S\|_g$, the \ALCH and \ALS conditions and Proposition \ref{prop:bound_nablaY_vY_u}, it finally holds that there exists $c>0$ such that
\begin{equation}\label{eq:bound_h}
|h| \leqslant c(\|u\|_g\|v\|_g + \|(\nabla_{Y_v}Y_u)(p)\|_g) e^{(2-\min\{a,b\})r}.
\end{equation}
Let $\alpha(p)$ be defined as $\alpha(p) = \frac{1}{2}\big(f(0,p) + \delr f(0,p) + \int_0^{+\infty} e^{-s}h(s,p)\dx s\big)$, which is well defined by \eqref{eq:bound_h}, and is continuous by the dominated convergence theorem applied on all compact subset where $u$ and $v$ are defined.
Hence,
\begin{equation*}\begin{split}
|f(r,p)-\alpha(p)e^r| &\leqslant \frac{e^r}{2}\int_r^{+\infty}e^{-s}|h(s,p)|\dx s \\&\quad + \frac{e^{-r}}{2}(|f(0,p)| + |\delr f(0,p)| + \int_0^r e^{s}|h(s,p)| \dx s).
\end{split}
\end{equation*}
Finally, notice that $|f(0,p)| = |g((\nabla_{Y_v}Y_u)(p),\Jdelr)| \leqslant \|(\nabla_{Y_v}Y_u)(p)\|_g$.
It now follows from \eqref{eq:delrf} that
\begin{equation*}\begin{split}
|\delr f(0,p)|  &= |   R(\nu(p),J\nu(p),v,u) + g(((\nabla_{Y_v}S)_p) u,\Jdelr)\\
&\quad + g(S(\nabla_{Y_v}Y_u)(p),\Jdelr) + g(JSu,Sv)|\\
&\leqslant \|R\|_g\|v\|_g\|u\|_g + \|(\nabla S)_p\|_g\|u\|_g\|v\|_g \\
&\quad + \|S\|_g\|(\nabla_{Y_v}Y_u)(p)\|_g + \|S\|_g\|u\|_g\|v\|_g.
\end{split}
\end{equation*}
The quantities $\|R\|_g$, $\|S\|_g$, and $\|\nabla S\|_g$ are uniformly bounded on the compact set $\dK$.
Moreover, \eqref{eq:bound_h} yields
\begin{equation*}
\int_r^{+\infty} e^{-s}|h(s,p)|\dx s \leqslant c(\|u\|_g\|v\|_g + \|(\nabla_{Y_v}Y_u)(p)\|_g) \frac{e^{(1-\min\{a,b\})r}}{\min\{a,b\}-1},
\end{equation*}
and
\begin{equation*}
\int_0^r e^s |h(s,p)|\dx s \leqslant c(\|u\|_g\|v\|_g + \|(\nabla_{Y_v}Y_u)(p)\|_g)\begin{cases}
\frac{e^{(3-\min\{a,b\})r}-1}{3-\min\{a,b\}} & \text{if } \min\{a,b\} \neq 3, \\
r & \text{otherwise}.
\end{cases}
\end{equation*}
The result then follows.
\end{proof}

We shall now prove our second Theorem, which we restate for the reader's convenience.

\begin{thmx}\label{thm:C1}
Let $(M,g,J)$ be an \ALCH and \ALS manifold of order $a$ and $b$, with an essential subset $K$, such that the sectional curvature of $\bar{M\setminus K}$ is negative.
If $\min\{a,b\}>1$, then the canonical $1$-form at infinity $\eta$ is a contact form of class $\mathcal{C}^1$, with Reeb vector field $\xi$.
\end{thmx}

\begin{proof}
To show that $\eta$ is of class $\mathcal{C}^1$, it suffices to fix a chart $U\subset \dK$ and to show that $\eta$ is of class $\mathcal{C}^1$ on $U$.
Let $m = 2n+1 = \dim \dK$, $\{x^1,\ldots,x^{m}\}$ be coordinates on $U\subset \dK$, and let $\{\partial_1,\ldots,\partial_{m}\}$ and $\{\dx x^1,\ldots, \dx x^{m}\}$ be the associated tangent frame and coframe.
For $r\geqslant 0$, the $1$-form $\eta_r$ and its partial derivatives locally read
\begin{equation*}
\eta_r  = \sum_{j=1}^{m} (\eta_r)_j \dx x^j \quad \text{and} \quad
\partial_i (\eta_r) = \sum_{i=1}^{m} \partial_i  (\eta_r)_j \dx x^j,
\end{equation*}
with $\partial_i (\eta_r)_j = (\mathcal{L}_{\partial_i} \eta_r)(\partial_j)$.
Write $Y_i = Y_{\partial_i}$ and $Y_j= Y_{\partial_j}$.
Lemma \ref{lem:L_uetav} then states that
\begin{equation*}
\forall i,j\in \{1,\ldots,m\},\quad 	\partial_i(\eta_r)_j = e^{-r}\big( g(\nabla_{Y_j}Y_i,\Jdelr) + g(\nabla_{Y_i}\Jdelr,Y_j)\big).
\end{equation*}
Proposition \ref{prop:convergence_to_alpha} now yields the existence of continuous functions $\alpha_{ij}\colon U \to \R$ and a constant $c>0$ such that if $i,j \in \{1,\ldots,m\}$, then
\begin{equation*}
|\partial_i(\eta_r)_j - \alpha_{ij}| \leqslant c(\|\partial_i\|_g\|\partial_j\|_g + \|\nabla_{\partial_j}\partial_i\|_g ) \begin{cases}
e^{(1-\min\{a,b\})r} & \text{if } 1<\min\{a,b\} < 3,\\
(r+1)e^{-2r} & \text{if } \min\{a,b\} = 3,\\
e^{-2r} & \text{if } \min\{a,b\} >3.
\end{cases}
\end{equation*}
The family $(\partial_i(\eta_r)_j)_{r\geqslant 0}$ then locally uniformly converges to the continuous function $\alpha_{ij}$ on $U$, and it follows that the canonical $1$-form $\eta$ is of class $\mathcal{C}^1$ on $U$, and hence on $\dK$.

We shall now show that $\eta$ is a contact $1$-form.
Since $(\eta_r)_{r\geqslant 0}$ locally converges to $\eta$ in $\mathcal{C}^1$ topology, it holds that
\begin{equation*}
\forall p \in \dK, \forall u,v \in T_p\dK,\quad \dx \eta|_p(u,v) = \lim_{r\to +\infty} \dx \eta_r|_p (u,v).
\end{equation*}
Let $r\geqslant 0$, $p\in \dK$ and $u,v \in T_p \dK$.
Consider smooth local extensions of $u$ and $v$ that are still denoted the same way.
Then
\begin{equation}\label{eq:deta_r}
\begin{split}
\dx \eta_r(u,v)  &=   u \cdot \eta_r(v) - v\cdot \eta_r(u) - \eta_r([u,v]) \\
&= e^{-r}(Y_u \cdot g(Y_v,\Jdelr) - Y_v\cdot g(Y_u,\Jdelr) - g([Y_u,Y_v],\Jdelr))\\
&= e^{-r}g(\nabla_{Y_u}Y_v - \nabla_{Y_v}Y_u - [Y_u,Y_v], \Jdelr)\\
& \quad + e^{-r}(g(Y_v,\nabla_{Y_u}\Jdelr) - g(Y_u,\nabla_{Y_v}\Jdelr))\\
&= e^{-r}(g(Y_v,JSY_u) - g(Y_u,JSY_v)).
\end{split}
\end{equation}
Let $\{\Jdelr,\delr,E_1,\ldots,E_{2n}\}$ be a radially parallel orthonormal frame on a cylinder $E(\R_+\times U)$ with $p\in U$.
Since $e^{-r}g(Y_v,JSY_u) \underset{r\to +\infty}{\longrightarrow} \frac{1}{2}\sum_{i,j=1}^{2n} \eta^i(v)\eta^j(u)g(E_i,\JE_i)$, it follows that
\begin{equation*}
\dx \eta(u,v) = \frac{1}{2}\sum_{i,j=1}^{2n} (\eta^i(v)\eta^j(u) - \eta^i(u)\eta^j(v))g(E_i,\JE_j).
\end{equation*}
Setting $\omega_{ij} = g(E_i,\JE_j)= g(e_i,Je_j)$, which are constants, the latter expression reads
\begin{equation}\label{eq:deta}
\dx \eta = -\frac{1}{2} \sum_{i,j=1}^{2n} \omega_{ij}\,\eta^i\wedge \eta^j.
\end{equation}
Assume furthermore that the local orthonormal frame $\{\nu,e_1,\ldots,e_{2n}\}$ on $U$ is chosen such that we have $Je_{2k-1}=e_{2k}$ for $k\in \{1,\ldots,n\}$.
In that case, the constants $\omega_{ij}$ are given by
\begin{equation*}
\omega_{ij} = \begin{cases}
-1 & \text{if } i=2k-1, j = 2k,\\
1 & \text{if } i= 2k, j= 2k-1,\\
0 & \text{otherwise}.
\end{cases}
\end{equation*}
Equation \eqref{eq:deta} then yields the equality
\begin{equation}\label{eq:wedge}
\dx \eta = \sum_{k=1}^{n} \eta^{2k-1}\wedge \eta^{2k}.
\end{equation}
From this last expression we derive the equality $(\dx\eta)^n =(n!)\, \eta^1\wedge\cdots \wedge \eta^{2n}$, and hence the equality $\eta \wedge (\dx \eta)^n =(n!)\, \eta \wedge \eta^1 \wedge \cdots \wedge \eta^{2n}$.
Since $\{\eta,\eta^1,\ldots,\eta^{2n}\}$ is a local coframe, $\eta\wedge (\dx\eta)^n$ is a volume form on $U$.
It follows that $\eta$ is a contact form.
Since $\dx \eta$ is a linear combination of wedge products of the local $1$-forms $(\eta^i)$, it follows by the very definition of $\xi$ that $\eta(\xi) = 1$ and $\dx \eta (\xi,\cdot) =0$, hence $\xi$ is the Reeb vector field of $\eta$.
The proof is now complete.
\end{proof}

\subsection{Regularity of the Carnot-Carath\'eodory metric}

We shall now study the regularity of the Carnot-Carathéodory metric $\gamma_H$.
To do so, we show that the local differential forms $\eta^1,\ldots, \eta^{2n}$, defined using a radially parallel orthonormal frame, have such regularity.
The condition $\min\{a,b\}>1$ is too weak to ensure that it is of class $\mathcal{C}^1$, since the renormalization of $\eta^j$ is of order $e^{-\frac{r}{2}}$ while that of $ \eta$ is of order $e^{-r}$.
We shall now show that the condition $\min\{a,b\}>\frac{3}{2}$ is sufficient.
The proof is very similar to that of Theorem \ref{thm:C1}.
We first give a bound on the growth of $\nabla_{Y_u}E_j$.

\begin{lemma}\label{lem:bound_nablaY_uX}
Let $(M,g,J)$ be an \ALCH manifold of order $a>\frac{1}{2}$ with an essential subset $K$.
Let $X$ be a radially parallel vector field on $\bar{M\setminus K}$, that is such that $\Ndr X = 0$ and assume that $\|X\|_g=1$.
Then there exists a constant $c>0$ such that
\begin{equation*}
\|\nabla_{Y_u}X\|_g \leqslant c \|u\|_ge^r.
\end{equation*}
\end{lemma}

\begin{proof}
Let $L = \|\nabla_{Y_u}X\|_g$, which is locally Lipschitz and hence almost everywhere differentiable.
Since $\Ndr X=0$, it holds that $\Ndr (\nabla_{Y_u}X) = -R(\delr,Y_u)X$, and therefore, applying $\delr$ to $\frac{1}{2}L^2$ gives the almost everywhere inequality
\begin{equation*}
L'L = -R(\delr,Y_u,X,\nabla_{Y_u}X) \quad a.e.
\end{equation*}
Recall that $\|R\|_g$ is uniformly bounded (Lemma \ref{lem:R_bounded}), and that since $a>\frac{1}{2}$, there exists $c_1>0$ such that $\|Y_v\|_g\leqslant c_1\|u\|_ge^r$ (Lemma \ref{lem:important_estimates}).
Hence, there exists $c>0$ such that
\begin{equation*}
L'L \leqslant c\|u\|_ge^r\|X\|_g L \quad a.e.
\end{equation*}
Since $L$ is non-positive and $\|X\|_g=1$, it follows that
\begin{equation*}
L' \leqslant c\|u\|_ge^r \quad a.e,
\end{equation*}
which is true even at points where $L$ vanishes since at those points, it achieves a minimum and its derivative thus vanishes.
The result follows from a straightforward integration.
\end{proof}
In particular, this Lemma applies to the vectors of a radially parallel orthonormal frame.

\begin{proposition}\label{prop:convergence_to_alphaj}
Let $(M,g,J)$ be an \ALCH and \ALS manifold of order $a$ and $b$, with $\min\{a,b\}>\frac{3}{2}$, with an essential subset $K$.
Let $\{\delr,\Jdelr,E_1,\ldots,E_{2n}\}$ be a radially parallel orthonormal frame on a cylinder $E(\R_+\times U)$.
Let $u$ and $v$ be local vector fields on $\dK$.
For all $j\in \{1,\ldots,2n\}$, define
\begin{equation*}
f^j(r,p)  = g(\nabla_{Y_v}Y_u,E_j) + g(\nabla_{Y_u}E_j,Y_v).
\end{equation*}
Then there exists a constant $c>0$ independent of $(r,p,u,v)$ and a continuous function $\alpha^j(p)$ on $U$ such that
\begin{equation*}
|f^j(r,p)-{e^\frac{r}{2}}\alpha^j(p) |\leqslant c(\|u\|_g\|v\|_g + \|\nabla_{Y_v}Y_u\|_g)\begin{cases}
e^{(2-\min\{a,b\})r} & \text{if } \min\{a,b\} < \frac{5}{2},\\
(r+1)e^{-\frac{r}{2}} & \text{if } \min\{a,b\} = \frac{5}{2},\\
e^{-\frac{r}{2}} & \text{if } \min\{a,b\}>\frac{5}{2}.
\end{cases}
\end{equation*}
\end{proposition}

\begin{proof}
The proof is a strict adaptation of that of Proposition \ref{prop:convergence_to_alpha}.
The exact same computations show that
\begin{equation*}\label{eq:messed_upj} \begin{split}
\delr \delr f^j &= (\Ndr R)(\delr,E_j,Y_v,Y_u) + R(\delr,E_j,SY_v,Y_u) \\
&\quad + R(\delr,E_j,Y_v,SY_u) -R(\delr,Y_v,SY_u,E_j)\\
&\quad  - (\nabla_{Y_v}R)(\delr,Y_u,\delr,E_j) - R(SY_v,Y_u,\delr,E_j)\\
&\quad -R(\delr,\nabla_{Y_v}Y_u,\delr,E_j) -R(\delr,Y_u,SY_v,E_j) \\
&\quad -R(\delr,Y_u,E_j,SY_v) -R(\delr,Y_v,\delr,\nabla_{Y_u}E_j),
\end{split}
\end{equation*}
and the exact same cancellations due to the symmetries of the Riemann tensor yield
\begin{equation}\label{eq:delrdelrfj}
\begin{split}
\delr\delr f^j &= (\Ndr R) (\delr,E_j,Y_v,Y_u) - (\nabla_{Y_v} R) (\delr,Y_u,\delr,E_j)\\
& \quad + R(\delr,E_j,Y_v,SY_u) - R(\delr,Y_v,SY_u,E_j)\\
& \quad -R(\delr,\nabla_{Y_v}Y_u,\delr,E_j) - R(\delr,Y_v,\delr,\nabla_{Y_u}E_j).
\end{split}
\end{equation}
Define $\mathbf{k}^j$ as
\begin{equation*}\begin{split}
\mathbf{k}^j &= R^0(\delr,E_j,Y_v,SY_u) - R^0(\delr,Y_v,SY_u,E_j)\\
& \quad -R^0(\delr,\nabla_{Y_v}Y_u,\delr,E_j) - R^0(\delr,Y_v,\delr,\nabla_{Y_u}E_j),
\end{split}
\end{equation*}
so that equation \eqref{eq:delrdelrfj} becomes
\begin{equation}\label{eq:RHS}\begin{split}
\delr\delr f^j - \mathbf{k}^j &= (\Ndr R) (\delr,E_j,Y_v,Y_u) - (\nabla_{Y_v} R) (\delr,Y_u,\delr,E_j)\\
& \quad + (R-R^0)(\delr,E_j,Y_v,SY_u) - (R-R^0)(\delr,Y_v,SY_u,E_j)\\
& \quad -(R-R^0)(\delr,\nabla_{Y_v}Y_u,\delr,E_j)\\&\quad  - (R-R^0)(\delr,Y_v,\delr,\nabla_{Y_u}E_j).
\end{split}
\end{equation}
From the computations of Lemma \ref{appendix:B2} in the Appendix, it turns out that
\begin{equation*}
\begin{split}
\mathbf{k}^j &= \frac{1}{4} g(SY_u,\Jdelr)g(Y_v,\JE_j) -\frac{1}{4} g(SY_u,\JE_j)g(Y_v,\Jdelr)\\
& \quad -\frac{1}{4}g(SY_u,\Jdelr)g(Y_v,\JE_j) - \frac{1}{2}g(SY_u,\JE_j)g(Y_v,\Jdelr)		\\
&\quad +\frac{1}{4}g(\nabla_{Y_v}Y_u,E_j) \\
&\quad +\frac{1}{4}g(Y_v,\nabla_{Y_u}E_j) + \frac{3}{4}g(SY_u,\JE_j)g(Y_v,\Jdelr)\\
&=\frac{1}{4}g(\nabla_{Y_v}Y_u,E_j) + \frac{1}{4}g(Y_v,\nabla_{Y_u}E_j)\\
&= \frac{1}{4}f^j.
\end{split}
\end{equation*}
Thus, $f^j$ is a solution to the second order linear ODE
\begin{equation*}
\delr\delr f^j - \frac{1}{4}f^j = h^j,
\end{equation*}
with $h^j$ being equal to the right-hand side of \eqref{eq:RHS}.
From classical second order linear ODE considerations, $f^j$ is given by
\begin{equation*}\label{eq:f^j} \begin{split}
f^j(r,p) &= e^{\frac{r}{2}}\left(\frac{1}{2}f^j(0,p) + \delr f^j(0,p) + \int_0^r e^{-\frac{s}{2}}h^j(s,p)\dx s\right) \\
&\quad + e^{-\frac{r}{2}}\left( \frac{1}{2}f^j(0,p) - \delr f^j(0,p) - \int_0^r e^{\frac{s}{2}}h^j(s,p)\dx s\right).
\end{split}
\end{equation*}
The following upper bound is straightforward
\begin{equation*}\begin{split}
|h^j| &\leqslant 2(\|R-R^0\|_g\|S\|_g + \|\nabla R\|_g)\|Y_u\|_g\|Y_v\|_g + \|R-R^0\|_g\|\nabla_{Y_v}Y_u\|_g \\
&\quad + \|R-R^0\|_g\|Y_v\|_g\|\nabla_{Y_u}E_j\|_g.
\end{split}
\end{equation*}
Lemma \ref{lem:bound_nablaY_uX} yields the existence of $c>0$ such that $\|\nabla_{Y_u}E_j\|_g\leqslant c\|u\|_ge^r$.
It now follows from the \ALCH and \ALS conditions together with Lemmas \ref{lem:important_estimates} and \ref{prop:bound_nablaY_vY_u} that there exists a constant $C>0$ such that
\begin{equation*}
|h^j| \leqslant C(\|u\|_g\|v\|_g + \|(\nabla_{Y_v}Y_u)(p)\|_g)e^{(2-\min\{a,b\})r},
\end{equation*}
and therefore, we have the following upper bounds
\begin{equation*}
\begin{split}
|e^{-\frac{s}{2}}h^j(s,p)| &\leqslant C(\|u\|_g\|v\|_g + \|(\nabla_{Y_v}Y_u)(p)\|_g)e^{(\frac{3}{2}-\min\{a,b\})r}, \\
|e^{\frac{s}{2}}h^j(s,p)| &\leqslant C(\|u\|_g\|v\|_g + \|(\nabla_{Y_v}Y_u)(p)\|_g)e^{(\frac{5}{2}-\min\{a,b\})r}.
\end{split}
\end{equation*}
Let $\alpha^j(p) = \frac{1}{2}f^j(0,p) + \delr f^j(0,p) + \int_0^{+\infty} e^{-\frac{s}{2}}h(s,p) \dx s $, which is continuous by the dominated convergence Theorem.
The result follows from a strictly similar study than that of the proof of Proposition \ref{prop:convergence_to_alpha}.
\end{proof}

We are now able to prove our third main result, which we restate for the reader's convenience.

\begin{thmx}\label{thm:gamma_H_C1}
Let $(M,g,J)$ be an \ALCH and \ALS manifold of order $a$ and $b$, with $\min\{a,b\}>\frac{3}{2}$, with an essential subset $K$ such that the sectional curvature of $\bar{M\setminus K}$ is negative.
Then the canonical Carnot-Carathéodory metric $\gamma_H$ has $\mathcal{C}^1$ regularity.
\end{thmx}

\begin{proof}
Let $\{\eta,\eta^1,\ldots,\eta^{2n}\}$ be the local coframe associated to a radially parallel orthonormal frame $\{\delr,\Jdelr,E_1,\ldots,E_{2n}\}$ on a cylinder $E(\R_+\times U)$.
Let $k\in \{1,\ldots,2n\}$ be fixed and $(\eta^k_r)_{r\geqslant 0}$ be the local family of $1$-forms that locally uniformly converges to $\eta^k$ in $\mathcal{C}^0$ topology.
Let $\{x^1,\ldots,x^{2n+1}\}$ be local coordinates on an open subset of $U$, and write $\eta_r^k = \sum_{j=1}^{2n+1} (\eta_r^k)_j \dx x^j$.
The partial derivatives of $\eta_r^k$ in these coordinates are given by
\begin{equation*}
\partial_i (\eta^k_r)(\partial_j) = e^{-\frac{r}{2}}\big( g(\nabla_{Y_j}Y_i,E_k) + g(\nabla_{Y_i}E_k,Y_j)\big),
\end{equation*}
and it follows from Proposition \ref{prop:convergence_to_alphaj} that they locally uniformly converges.
Hence, $\eta^k$ is of class $\mathcal{C}^1$.

The local coframe $\{\eta,\eta^1,\ldots,\eta^{2n}\}$ has then been shown to be of class $\mathcal{C}^1$.
Since the Carnot-Carathéodory metric is locally given by
$
\gamma_H = \sum_{i=1}^{2n}\eta^i\otimes \eta^i
$
it follows that it has $\mathcal{C}^1$ regularity.
This concludes the proof.
\end{proof}
\begin{remark}
It is worth noting that in that case, although $\eta$ is not of class $\mathcal{C}^2$, its exterior differential $\dx \eta$ is of class $\mathcal{C}^1$, since it is locally expressed as a combination of $\{\eta^i\wedge \eta^j\}_{i,j\in \{1,\ldots,2n\}}$.
\end{remark}


\section{The almost complex structure}
\label{section:almost_complex_structure}

\subsection{Notations}

Throughout this section, the K\"ahler manifold $(M,g,J)$ is assumed to satisfy the \ALCH and \ALS conditions of orders $a,b>\frac{3}{2}$, with an essential subset $K$ such that the sectional curvature of $\bar{M\setminus K}$ is negative.
From Theorem \ref{thm:gamma_H_C1}, $\dK$ is endowed with a contact form $\eta$ of class $\mathcal{C}^1$ with contact structure $H = \ker \eta$, and with a $\mathcal{C}^1$ Carnot-Carathéodory metric $\gamma_H$ which is positive definite on $H$.
Let $g_0 = g|_{\dK}$ be the induced metric on $\dK$, and recall that if $g_r = E(r,\cdot)^*g|_{\{\delr\}^{\perp}}$, Corollary \ref{cor:norm_Yu_greater} yields the existence of $c_0>0$ such that $g_0 \leqslant  c_0e^{-r}g_r$.

For $r \geqslant 0$, we set $S_r = E(r,\cdot)^* S$ the pull-back of the shape operator thought as a field of endomorphisms of $\{\delr\}^{\perp}$.
To any radially parallel orthonormal frame $\{\delr,\Jdelr,E_1,\ldots,E_{2n}\}$ is associated a local coframe $\{\eta_r,\eta^1_r,\ldots,\eta^{2n}_r\}$ which locally converges in $\mathcal{C}^1$ topology to the local coframe $\{\eta,\eta^1,\ldots,\eta^{2n}\}$.
Moreover, there exists a constant $c>0$ independent of the radially parallel orthonormal frame and from $r$, such that the differential forms $\{\eta_r,\eta^1_r,\ldots,\eta_r^{2n}\}$ and $\{\eta,\eta^1,\ldots,\eta^{2n}\}$ all have $g_0$-norm less than $c$.
By Proposition \ref{prop:convergence_etar}, it holds that
\begin{equation}\label{eq:eta-eta_r}
\max\{ e^r\|\eta_r-\eta\|_{g_0}, e^{\frac{r}{2}}\|\eta^1_r-\eta^1\|_{g_0},\ldots,e^{\frac{r}{2}}\|\eta^{2n}_r-\eta^{2n}\|_{g_0}\} \leqslant c e^{-\frac{r}{2}}.
\end{equation}
If $v$ is a vector field tangent to $\dK$ and if $Z_v$ and $Z'_v$ are the vector fields asymptotic to $Y_v$ and $SY_v$, then Proposition \ref{prop:Y_v-Z_v} states that there exists $C>0$ such that
\begin{equation}\label{eq:z_v-y_v_normalized}
\max\{\|Y_v-Z_v\|_g,\|SY_v-Z'_v\|_g\} \leqslant C \|v\|_g e^{-\frac{r}{2}}.
\end{equation}
The dual frames of the local coframes $\{\eta_r,\eta^1_r,\ldots,\eta_r^{2n}\}$ and $\{\eta,\eta^1,\ldots,\eta^{2n}\}$ are denoted by $\{\xi^r,\xi^r,\ldots,\xi^r_{2n}\}$ and $\{\xi,\xi_1,\ldots,\xi_{2n}\}$.
Remark that $\xi^r = E(r,\cdot)^* \left(e^r\Jdelr\right)$ and $\xi^r_j = E(r,\cdot)^* \left(e^{\frac{r}{2}}E_j\right)$.
Since $\{\eta_r,\eta^1_r,\ldots,\eta_r^{2n}\}$ locally converges in $\mathcal{C}^1$ topology, so does $\{\xi^r,\xi^r,\ldots,\xi^r_{2n}\}$.
It follows that $\{\xi,\xi_1,\ldots,\xi_{2n}\}$ is of class $\mathcal{C}^1$.

Let us define a field of endomorphisms on $\dK$ related to the ambient almost complex structure $J$.
Since $J$ does not preserve $\{\delr\}^{\perp}$, it is not possible to pull it back on $\dK$ by $E(r,\cdot)$.
However, the tensor
\begin{equation*}
\Phi = J - g(\cdot,\delr)\otimes \Jdelr + g(\cdot,\Jdelr) \otimes \delr,
\end{equation*}
stabilizes the distribution $\{\delr\}^{\perp}$.

\begin{definition}
For $r\geqslant 0$, let $\phi_r$ be defined by $\phi_r = E(r,\cdot)^* \Phi$.
\end{definition}
Since $\Phi$ satisfies the two equalities $\Phi^2 = -\Id + g(\cdot,\delr)\otimes \delr + g(\cdot, \Jdelr)\otimes \Jdelr$ and $\Phi^3 = - \Phi$, we have the immediate Lemma.

\begin{lemma}\label{lem:almost_constact}
For all $r \geqslant 0$, $\phi_r^2 = -\Id + \eta_r\otimes \xi^r$, and $\phi_r^3 = -\phi_r$.
\end{lemma}

\begin{remark}
Lemma \ref{lem:almost_constact} states that $(\dK,\phi_r,\eta_r,\xi^r)$ is an \textit{almost contact} manifold in the sense of \cite{sasaki_differentiable_1960,sasaki_differentiable_1961,blair_riemannian_2010}.
The following estimates are worth noting.
\end{remark}

\begin{lemma}\label{lem:bahaviour_phi_r}
There exists $c_1,c_2>0$ such that for all $r \geqslant 0$, it holds
\begin{enumerate}
\item $\|\phi_r\|_{g_0} \leqslant 1$,
\item $\|\phi_r \xi\|_{g_0} \leqslant c_1 e^{-r}$,
\item $\|\eta\circ \phi_r\|_{g_0} \leqslant c_2e^{-r}$.
\end{enumerate}
\end{lemma}

\begin{proof}
\begin{enumerate}
\item Since $\Phi = \pi^{\perp}\circ J \circ \pi^{\perp}$, where $\pi^{\perp}$ is the orthogonal projection on ${\delr}^{\perp}$ and $J$ is an isometry, $\Phi$ has an operator norm less than or equal to $1$.
The result follows.
\item Since $\phi_r \xi^r =0$, it holds that
\begin{equation*}\begin{split}
\|\phi_r\xi\|_{g_0} &= \|\phi_r(\xi - \xi^r)\|_{g_0} \\
&\leqslant \|\phi_r\|_g\|\xi-\xi^r\|_{g_0} \\
&\leqslant c_0e^{-\frac{r}{2}}\|Y_{\xi} - e^r\Jdelr\|_{g_0} \\
&\leqslant c_0C\|\xi\|_{g_0} e^{-r},
\end{split}
\end{equation*}
the inequalities being derived from the first point and from equation \eqref{eq:z_v-y_v_normalized}.
To conclude, define $c_1 = c_0C \sup_{\dK}\|\xi\|_g <+\infty$.
\item Since $\eta_r \circ \phi_r = 0$, it holds that
\begin{equation*}
\|\eta\circ \phi_r\|_{g_0} = \|(\eta-\eta_r)\circ \phi_r\|_{g_0} \leqslant \|\eta-\eta_r\|_{g_0} \|\phi_r\|_{g_0} \leqslant ce^{-\frac{3}{2}r},
\end{equation*}
the latter inequality being derived from equation \eqref{eq:eta-eta_r} and from the first point.
The result is then true with $c_2=c$.
\end{enumerate}
\end{proof}

\subsection{Shape operator estimates}

We shall now give estimates on the asymptotic behaviour of $(S_r)_{r\geqslant 0}$.
We first prove that $S_r$ is asymptotic to the tensor $\frac{1}{2}(\Id+\eta \otimes \xi)$.

\begin{lemma}\label{lem:s_r_asymptotic}
There exists a constant $c'>0$ such that
\begin{equation*}
\forall r \geqslant 0,\quad  \|S_r - \frac{1}{2}(\Id + \eta\otimes \xi) \|_{g_0} \leqslant c' e^{-r}.
\end{equation*}
\end{lemma}

\begin{proof}
Fix $v \in T\dK$.
From Corollary \ref{cor:norm_Yu_greater}, there exists $c_0>0$ such that
\begin{equation}\label{eq:norm_S_r}
\|S_rv - \frac{1}{2}(v + \eta(v)\xi)\|_{g_0}\leqslant c_0 e^{-\frac{r}{2}}\|SY_v - \frac{1}{2}(Y_v + \eta(v) Y_{\xi})\|_{g_0}.
\end{equation}
The local expressions of $Z_v$ and $Z_v'$ in a radially parallel orthonormal frame show that $\frac{1}{2}\eta(v)e^r\Jdelr = Z_v' - \frac{1}{2}Z_v$.
Therefore, it holds that
\begin{equation*}
SY_v - \frac{1}{2}(Y_v+ \eta(v)e^r\Jdelr) = SY_v - Z_v'  - \frac{1}{2}(Y_v - Z_v) + \frac{1}{2}\eta(v)(e^r\Jdelr - Y_{\xi}).
\end{equation*}
By definition, $\eta(\xi) =1$, and thus $e^r\Jdelr = Z_{\xi}$.
It follows from the triangle inequality that
\begin{equation*}\begin{split}
\|SY_v - \frac{1}{2}(Y_v+ \eta(v)e^r\Jdelr)\|_g & \leqslant \|SY_v-Z_v'\|_g + \frac{1}{2}\|Y_v-Z_v\|_g \\&\quad + \frac{1}{2} |\eta(v)| \|Y_{\xi}-Z_{\xi}\|_g.
\end{split}
\end{equation*}
It now follows from the uniform bound on $\|\eta\|_{g_0}$ and from equation \eqref{eq:z_v-y_v_normalized} that there exists $C>0$ such that
\begin{equation*}
\|SY_v - \frac{1}{2}(Y_v+ \eta(v)e^r\Jdelr)\|_g \leqslant C\|v\|_g(1+\|\xi\|_g)e^{-\frac{r}{2}}.
\end{equation*}
Finally, equation \eqref{eq:norm_S_r} yields
\begin{equation*}
\|S_rv - \frac{1}{2}(v + \eta(v)\xi)\|_{g_0}\leqslant c_0C(1+\|\xi\|_{g_0}) \|v\|_{g} e^{-r}.
\end{equation*}
The result follows by setting $c' = c_0C(1+\sup_{\dK}\|\xi\|_{g_0})$, which is finite since $\xi$ is a global vector field and $\dK$ is compact.
\end{proof}

\begin{remark}
In the model setting, the shape operator of concentric spheres is of the form $S = \coth r \Id_{\R \Jdelr} + \frac{1}{2}\coth(\frac{r}{2})\Id_{\{\delr,\Jdelr\}^{\perp}}$.
\end{remark}
We shall now show that $S_r$ and $\phi_r$ asymptotically commute.

\begin{lemma}\label{lem:s_r-phi_r_commute}
There exists $c''>0$ such that
\begin{equation*}
\forall r \geqslant 0,\quad \|S_r\phi_r - \phi_r S_r \|_{g_0} \leqslant c'' e^{-r}.
\end{equation*}
\end{lemma}

\begin{proof}
First, write
\begin{equation*}
\begin{split}
S_r\phi_r &= (S_r - \frac{1}{2}(\Id + \eta\otimes \xi)) \phi_r + \frac{1}{2}\phi_r + \frac{1}{2} (\eta\circ \phi_r) \otimes \xi,\\
\phi_r S_r&= \phi_r(S_r - \frac{1}{2}(\Id + \eta\otimes \xi)) + \frac{1}{2}\phi_r + \frac{1}{2} \eta \otimes (\phi_r \xi).
\end{split}
\end{equation*}
It now follows from the triangle inequality that
\begin{equation*}
\|S_r\phi_r - S_r\phi_r \|_{g_0} \leqslant 2\|\phi_r\|_{g_0} \|S_r - \frac{1}{2}(\Id + \eta\otimes \xi)\|_{g_0}
+ \|\eta\circ \phi_r\|_{g_0} \|\xi\|_{g_0} + \|\eta\|_{g_0} \|\phi_r \xi\|_{g_0}.
\end{equation*}
The result follows from Lemmas \ref{lem:bahaviour_phi_r} and \ref{lem:s_r_asymptotic}.
\end{proof}

\subsection{Convergence}

We now show that $(\phi_r)_{r\geqslant 0}$ converges in $\mathcal{C}^1$ topology to some tensor $\phi$, and that the restriction of this tensor to $H$ is an almost complex structure.

\begin{proposition}\label{prop:almost_complex}
The family $(\phi_r)_{r\geqslant 0}$ converges in $\mathcal{C}^1$ topology to a $\mathcal{C}^1$ tensor $\phi$ satisfying $\phi^2 = -\Id + \eta\otimes \xi$, $\phi^3 = -\phi$, $\eta\circ\phi  = 0$ and $\eta\xi = 0$.
In particular, $\phi$ preserves $H = \ker \eta$ and $({\phi|_H})^2 = - \Id_H$.
\end{proposition}

\begin{proof}
Let $\{\delr,\Jdelr,E_1,\ldots,E_{2n}\}$ be a radially parallel orthonormal frame on a cylinder $E(\R_+\times U)$.
By the very definition of $\Phi$, it holds that
\begin{equation*}
\Phi = \sum_{j=1}^{2n} g(\cdot,E_j)\otimes \JE_j.
\end{equation*}
Assume that the radially parallel orthonormal frame is a $J$-frame, that is for any $k\in \{1,\ldots,n\}$, it holds that $\JE_{2k-1} = \JE_{2k}$.
Then $\Phi$ has the nice expression
\begin{equation*}
\Phi = \sum_{k=1}^n g(\cdot,E_{2k-1})\otimes E_{2k} - g(\cdot,E_{2k})\otimes E_{2k-1},
\end{equation*}
from which we deduce the following expression for $\phi_r$
\begin{equation*}
\forall r \geqslant 0,\quad \phi_r = \sum_{k=1}^n \eta^{2k-1}_r \otimes \xi^r_{2k} - \eta^{2k}_r \otimes \xi^r_{2k-1}.
\end{equation*}
From the convergence of the local coframe $\{\eta_r,\eta^1_r,\ldots,\eta_r^{2n}\}$ to $\{\eta,\eta^1,\ldots,\eta^{2n}\}$ and of the local frame $\{\xi^r,\xi_1^r,\ldots,\xi_{2n}^r\}$ to $\{\xi,\xi_1,\ldots,\xi_{2n}\}$ in $\mathcal{C}^1$ topology, it follows that on $U$, $(\phi_r)_{r\geqslant 0}$ converges in $\mathcal{C}^1$ topology to
\begin{equation*}
\phi = \sum_{k=1}^n \eta^{2k-1}\otimes \xi_{2k} - \eta^{2k}\otimes \xi_{2k-1}.
\end{equation*}
Notice that this does not depend on the chosen radially parallel $J$-orthonormal frame, and $\phi_r\to \phi$ in $\mathcal{C}^1$ topology on $\dK$.
Taking the limit as $r\to +\infty$ in Lemmas \ref{lem:almost_constact} and \ref{lem:bahaviour_phi_r} concludes the proof.
\end{proof}
\begin{definition}\label{dfn:phi}
The almost complex structure $J_H$ on $H$ is defined as the restriction of $\phi$ to $H$.
\end{definition}

\subsection{Integrability}

Since $J_H$ is an almost complex structure, the complexified bundle $H\otimes \C$ splits as $H\otimes \C = H^{1,0} \oplus H^{0,1}$, where $H^{1,0} = \ker \{J_H - i \Id \}= \{X+iJ_HX \mid X\in H\}$ and $H^{0,1} = \ker \{ J_H + i \Id\}$.
Since $H$ and $J_H$ are both of class $\mathcal{C}^1$, the Lie bracket of sections of $H^{1,0}$ makes sense.
It is then possible to ask whether $J_H$ is integrable, that is,  if $H^{1,0}$ is stable under the Lie bracket.
However, $J_H$ is defined as the restriction of the limit of a family of tensors $(\phi_r)$, which does not preserve $H$, and  it is not clear what condition on $\phi_r|_{\ker \eta_r}$ would ensure that $J_H$ is integrable.
Considering the whole tensor $\phi$ is thus more convenient, as the following study shows.

Recall that	$\phi$ satisfies $\phi^3 = -\phi$.
The complexified tangent bundle $T\dK\otimes \C$ then splits into the direct sum of the eigenspaces of $\phi$ as
\begin{equation*}
T\dK\otimes \C = \ker \phi \oplus \ker \{\phi - i\Id \} \oplus \ker \{\phi + i \Id\} = \C \xi \oplus H^{1,0} \oplus H^{0,1}.
\end{equation*}
We still denote by $\phi$ and $\eta$ the complex-linear extensions of $\phi$ and $\eta$.
Any complex vector field $V \in \Gamma(T\dK)\otimes \C$ reads
$V = \eta(V)\xi + V^{1,0} + V^{0,1}$,
with $\phi V^{1,0}=iV^{1,0}$ and $\phi V^{0,1}=-iV^{0,1}$.
Since $\phi \xi = 0$, it follows that $V+i\phi V = \eta(V)\xi + 2 V^{0,1}$.
Hence, a complex vector field $V$ on $\dK$ is a section of $H^{1,0}$ if and only if $V+i\phi V = 0$.
Finally, it follows that $H^{1,0}$ is integrable if and only if
\begin{equation*}
\forall u,v \in \Gamma(H), \quad [u-i\phi u,v-i\phi v] + i \phi [u-i\phi u, v-i\phi v] = 0.
\end{equation*}

The Nijenhuis tensor $N_A$ of a field of endomorphisms $A$ is defined by
\begin{equation*}
\forall X,Y,\quad N_A(X,Y) = -A^2[X,Y] - [AX,AY] + A[AX,Y] + A[X,AY] .
\end{equation*}
The integrability of $J_H$ is related to the Nijenhuis tensor of $\phi$ by the following Lemma.
\begin{lemma}\label{lem:nijenhuis}
For any vector fields $u$ and $v$ on $\dK$, setting $V = [u-i\phi u, v-i\phi v]$,  it holds that
\begin{equation*}
\begin{split}
V+i\phi V &= N_{\phi}(u,v) + \eta([u,v])\xi + i\phi N_{\phi}(u,v)- i (\eta([\phi u,v]) + \eta([u,\phi v])) \xi.
\end{split}
\end{equation*}
\end{lemma}

\begin{proof}
Extending the Lie bracket $\C$-linearly, it holds that
\begin{equation*}
\begin{split}
V +i\phi V &= [u,v] - [\phi u, \phi v] -i [\phi u,v] - i [u,\phi v]\\
&\quad + i \phi \big( [u,v] - [\phi u, \phi v] -i [\phi u,v] - i [u,\phi v]\big)\\
&= [u,v] - [\phi u, \phi v] + \phi[\phi u, v] + \phi[ u,\phi v]\\
&\quad + i \big( \phi [u,v] - \phi[\phi u, \phi v] -[\phi u,v] - [u, \phi v] \big).
\end{split}
\end{equation*}
The equalities $\phi^2 = -\Id + \eta\otimes \xi$ and $\phi^3 = -\phi$ now yield
\begin{equation*}\begin{split}
V+i\phi V &= (-\phi^2 + \eta\otimes \xi)[u,v]  - [\phi u, \phi v] + \phi[\phi u, v] + \phi[ u,\phi v]\\
&\quad + i \big( -\phi^3[u,v] - \phi[\phi u, \phi  v] + (\phi^2 - \eta\otimes \xi)[\phi u, v] \\
&\quad + (\phi^2 - \eta\otimes \xi)[u,\phi v] \big) \\
&= -\phi^2[u,v] -[\phi u, \phi v] + \phi[\phi u, v] + \phi [u,\phi v] + \eta([u,v])\xi\\
& \quad + i \phi\big( -\phi^2[u,v] -[\phi u, \phi v] + \phi[\phi u, v] + \phi [u,\phi v]\big)\\
& \quad - i\big( \eta([\phi u,v]) + \eta([u,\phi v]) \big) \xi.
\end{split}
\end{equation*}
The result follows from the definition of the Nijenhuis tensor $N_{\phi}$.
\end{proof}

\begin{proposition}\label{prop:integrability_conditions}
The almost complex structure $J_H$ is integrable if and only if the two following conditions are satisfied:
\begin{enumerate}
\item $N_{\phi}|_{H\times H} = \dx \eta|_{H\times H} \otimes \xi$,
\item $\dx\eta|_{H\times H}(J_H\cdot,\cdot) = -\dx\eta|_{H\times H}(\cdot,J_H \cdot)$.
\end{enumerate}
\end{proposition}

\begin{proof}
Let $u$ and $v$ be tangent to the distribution $H$ and $V=[u-i\phi u, v-i \phi v]$.
Then, $V$ is tangent to $H^{1,0}$ if and only if  $V+ i \phi V = 0$.
Identifying the real and imaginary parts, Lemma \ref{lem:nijenhuis} then states that $V$ is tangent to $H^{1,0}$ if and only if
\begin{equation*}
\begin{cases}
0&=N_{\phi}(u,v) + \eta([u,v])\xi, \\
0&=\phi N_{\phi}(u,v) - \big(\eta([\phi u, v]) + \eta([u,\phi v])\big) \xi.
\end{cases}
\end{equation*}
Since $\phi \xi = 0$, this latter system is equivalent to
\begin{equation*}\begin{cases}
N_{\phi}(u,v) &=- \eta([u,v])\xi, \\
\eta([\phi u, v]) &= - \eta([u,\phi v]).
\end{cases}
\end{equation*}
By definition of $u$ and $v$, $\eta(u)=\eta(v) = 0$.
Moreover, since $\eta\circ \phi = 0$, it follows that $\eta(\phi u) = \eta(\phi v) = 0$.
Therefore, it holds that
\begin{equation*}
\begin{cases}
\dx \eta([u,v]) = u \cdot \eta(v) - v\cdot \eta(u) - \eta([u,v] = -\eta([u,v])),\\
\dx\eta ([\phi u, v]) = (\phi u) \cdot \eta(v) - v \cdot \eta(\phi u) - \eta([\phi u,v]) = -\eta([\phi u, v]),\\
\dx\eta ([ u,\phi v]) =  u \cdot \eta(\phi v) - (\phi v) \cdot \eta( u) - \eta([u,\phi v]) = -\eta([ u,\phi v]).
\end{cases}
\end{equation*}
Finally, $V$ is tangent to $H^{1,0}$ if and only if
$
N_{\phi}(u,v) = \dx \eta(u,v) \xi$ and $\dx \eta (\phi u, v) = - \dx \eta (u, \phi v)$.
\end{proof}

\begin{lemma}\label{lem:almost_type_(1,1)}
There exists $\tilde{c}>0$ such that for all $r\geqslant0$, $p\in \dK$ and $u,v\in T_p\dK$, we have
\begin{equation*}
|\dx \eta_r(\phi_r u,v) + \dx \eta_r(u,\phi_r v)| \leqslant \tilde{c}\|u\|_g\|v\|_g e^{-\frac{r}{2}}.
\end{equation*}
\end{lemma}

\begin{proof}
Recall that equation \eqref{eq:deta_r} gives, for all $r \geqslant 0$ and $u,v$ tangent to $\dK$
\begin{equation*}
\dx \eta_r(u,v) = e^{-r} \big( g(Y_v,JSY_u) - g(Y_u,JSY_v)\big).
\end{equation*}
Since $J$ is skew-symmetric, it holds that
\begin{equation*}
\dx\eta_r(u,v) = e^{-r}\big( g(SY_u,JY_v) - g(SY_v,JY_u)\big).
\end{equation*}
Notice that since $Y_u$ and $Y_v$ are normal to $\delr$, this latter expression reads
\begin{equation*}
\dx\eta_r(u,v) = e^{-r}\big( g(SY_u,\Phi Y_v) - g(SY_v,\Phi Y_u)\big).
\end{equation*}
It thus holds that
\begin{equation*}\begin{split}
\dx \eta_r(\phi_ru,v) &= e^{-r}\big( g(S\Phi Y_u,\Phi Y_v) - g(SY_v,\Phi^2 Y_u)\big),\\
\dx\eta_r(u,\phi_r v) &= e^{-r}\big( g(SY_u,\Phi^2 Y_v) - g(S\Phi Y_v,\Phi Y_u)\big).
\end{split}
\end{equation*}
The symmetry of $S$ now yields
\begin{equation*}
\dx\eta_r(\phi_ru,v) + \dx\eta_r(u,\phi_r v) = e^{-r} \big(g(SY_u,\Phi^2Y_v) - g(SY_v,\Phi^2 Y_u)  \big).
\end{equation*}
Since $\Phi^2Y_u = -Y_u + g(Y_u,\Jdelr)\otimes \Jdelr$, and similarly for $Y_v$, it follows that
\begin{equation*}\begin{split}
\dx\eta_r(\phi_ru,v) + \dx\eta_r(u,\phi_r v) &= e^{-r} \big(g(SY_v,\Jdelr)g(Y_u,\Jdelr)
\\& \quad - g(SY_u,\Jdelr)g(Y_v,\Jdelr) \big).
\end{split}
\end{equation*}
Using the symmetry of $S$ and writing $S\Jdelr = (S\Jdelr - \Jdelr) + \Jdelr$ now shows that
\begin{equation*}
\dx\eta_r(\phi_ru,v) + \dx\eta_r(u,\phi_r v) = \eta_r(u)g(Y_v,S\Jdelr-\Jdelr) - \eta_r(v)g(Y_u,S\Jdelr - \Jdelr).
\end{equation*}

Recall that there exist constants $c_1,c_2>0$ such that $\|\eta_r\|_g \leqslant c_1$ and $\|Y_w\|_g \leqslant c_2\|w\|_ge^r$ for any $w$.
Thus
\begin{equation*}
|\dx\eta_r(\phi_ru,v) + \dx\eta_r(u,\phi_r v)| \leqslant c\|u\|_g\|v\|_g \|Se^r\Jdelr - e^r\Jdelr\|_g.
\end{equation*}
To conclude, it suffices to control $\|Se^r\Jdelr - e^r\Jdelr\|_g$.
Recall that $e^r \Jdelr = Z_{\xi} = Z'_{\xi}$.
It now follows from the triangle inequality, the fact that $S$ is uniformly bounded, and equation \eqref{eq:z_v-y_v_normalized} that there exists $c_3 >0$ such that
\begin{equation*}\begin{split}
\|Se^r\Jdelr - e^r\Jdelr\|_g &= \|S(Z_{\xi} - Y_{\xi}) + SY_{\xi} - Z_{\xi}'\|_g\\
& \leqslant \|S\|_g\|Y_{\xi}-Z_{\xi}\|_g + \|SY_{\xi} - Z'_{\xi}\|_g\\
& \leqslant c_3 \|\xi\|_ge^{-\frac{r}{2}}.
\end{split}
\end{equation*}
The proof follows by setting $\tilde{c}=c_1c_2c_3 \sup_{\dK}\|\xi\|_g < + \infty$ since $\xi$ is a continuous vector field on $\dK$ compact.
\end{proof}

\begin{proposition}\label{prop:J_H_is_(1,1)}
Let $u$ and $v$ be vector fields on $\dK$.
Then it holds that $\dx \eta(\phi u,v) = -\dx \eta(u,\phi v)$.
In particular, $\dx \eta|_{H\times H}(J_H\cdot,\cdot) = -\dx \eta|_{H\times H}(\cdot,J_H\cdot)$.
\end{proposition}

\begin{proof}
Recall that $(\phi_r)_{r\geqslant 0}$ and $(\dx\eta_r)_{r\geqslant 0}$ locally uniformly converge to $\phi$ and $\dx \eta$.
The result follows by taking the limit in Lemma \ref{lem:almost_type_(1,1)} as $r\to +\infty$.
\end{proof}

Proposition \ref{prop:J_H_is_(1,1)} shows that the condition (2) of Proposition \ref{prop:integrability_conditions} is satisfied.
Let us show that the condition (1) is also satisfied.

\begin{lemma}\label{prop:nijenhuis_r}
For all $r\geqslant 0$, $p\in \dK$ and $u,v\in T_p\dK$, we have
\begin{equation*}\begin{split}
N_{\phi_r}(u,v) &= e^r\eta_r(v)(\phi_r S_r u - S_r\phi_r u) - e^r\eta_r(v)(\phi_r S_r v - S_r \phi_r v)\\
&\quad   + \dx \eta_r(u,v)\xi(r).
\end{split}
\end{equation*}
\end{lemma}

\begin{proof}
For $X$ and $Y$ vector fields on $\bar{M\setminus K}$, it holds that
\begin{equation*}
\begin{split}
-\Phi^2[X,Y] &= -\Phi^2 \nabla_XY + \Phi^2 \nabla_YX,\\
-[\Phi X,\Phi Y] &= -\nabla_{\Phi X}(\Phi Y) + \nabla_{\Phi Y}(\Phi X)\\
&= -(\nabla_{\Phi X}\Phi)Y - \Phi \nabla_{\Phi X}Y + (\nabla_{\Phi Y}\Phi)X + \Phi \nabla_{\Phi Y}X,\\
\Phi[\Phi X,Y] &= \Phi \nabla_{\Phi X} Y - \Phi\nabla_Y (\Phi X)\\
&= \Phi \nabla_{\Phi X} Y - \Phi (\nabla_Y\Phi)X - \Phi^2 \nabla_{X}Y, \text{ and}\\
\Phi[X,\Phi Y] &= \Phi \nabla_X(\Phi Y) - \Phi \nabla_{\Phi Y} X\\
&= \Phi(\nabla_X\Phi)Y + \Phi^2\nabla_XY - \Phi \nabla_{\Phi Y}X.
\end{split}
\end{equation*}
Recall that $N_{\Phi}(X,Y) = -\Phi^2[X,Y] -[\Phi X, \Phi Y] + \Phi[\Phi X, Y] + \Phi [X,\Phi Y]$.
Therefore,
\begin{equation*}
N_{\Phi}(X,Y) = \Phi(\nabla_{X}\Phi)Y - (\nabla_{\Phi X}\Phi)Y  + (\nabla_{\Phi Y}\Phi)X - \Phi(\nabla_{Y}\Phi)X.
\end{equation*}
Recall that $\nabla g = 0$, $\nabla J = 0$, $\nabla \delr = S$ and $\nabla \Jdelr = J S$.
Hence, by the very definition of $\Phi$, it holds that
\begin{equation*}
\begin{split}
(\nabla_{X}\Phi) Y &= -g(Y,SX) \Jdelr - g(Y,\delr)JSX\\
&\quad + g(Y,JSX) \delr + g(Y,\Jdelr) SX.
\end{split}
\end{equation*}
Since $\Phi\delr = \Phi \Jdelr = 0$, it holds that
\begin{equation*}
\Phi (\nabla_X\Phi) Y = -g(Y,\delr)\Phi JSX + g(Y,\Jdelr)\Phi SX.
\end{equation*}
Similarly, the following equality holds
\begin{equation*}
\begin{split}
(\nabla_{\Phi X}\Phi)Y &= -g(Y,S\Phi X) \Jdelr - g(Y,\delr)JS\Phi X \\
&\quad + g(Y,JS\Phi X)\delr + g(Y,\Jdelr)S\Phi X.
\end{split}
\end{equation*}
It follows by the anti-symmetry in $X$ and $Y$ that
\begin{equation}\label{eq:third_fourth_lines} \begin{split}
N_{\Phi}(X,Y) &= g(Y,\Jdelr)( \Phi SX - S\Phi X) - g(X,\Jdelr) (\Phi SY -S\Phi Y)\\
& \quad +(g(Y,S\Phi X) - g(X,S\Phi Y)) \Jdelr\\
& \quad + (g(Y,JS\Phi X) - g(X,JS\Phi Y)) \delr\\
& \quad - g(Y,\delr) JS\Phi X + g(X,\delr) JS\Phi Y.
\end{split}
\end{equation}
Notice that $g(X,JS\Phi Y) = -g(JX, S\Phi Y)$.
Assume from now that $X$ is orthogonal to $\delr$.
Then $JX$ is orthogonal to $\Jdelr$.
Since $S\Phi Y$ is orthogonal to $\delr$, it follows that $g(JX, S\Phi Y) = g(\Phi X,S\Phi Y)$.
In particular, if $X$ and $Y$ are both orthogonal to $\delr$, the terms in the third and fourth lines of equation \eqref{eq:third_fourth_lines} vanish.
Hence, if $u$ and $v$ are vector fields on $\dK$, it holds that
\begin{equation*}\begin{split}
N_{\Phi}(Y_u,Y_v) &= g(Y_v,\Jdelr)(\Phi S Y_u - S\Phi Y_u) - g(Y_u,\Jdelr)(\Phi S Y_v - S\Phi Y_v)\\
& \quad + (g(Y_v,S\Phi Y_u) - g(Y_u,S\Phi Y_v)) \Jdelr.
\end{split}
\end{equation*}
Recall from equation \eqref{eq:deta_r} that $\dx \eta_r(u,v) = e^{-r}(g(Y_v,JSY_u) - g(Y_u,JSY_v))$.
In addition, notice that
\begin{equation*}
g(Y_v,JSY_u) = -g(JY_v,SY_u) = -g(SJY_v,Y_u) = -g(S\Phi Y_v,Y_u).
\end{equation*}
Hence, $g(Y_v,S\Phi Y_u) - g(Y_u,S\Phi Y_v) = e^r\dx\eta_r(u,v)$.
The result now follows from the equality $E(r,\cdot)^* e^r\Jdelr = \xi(r)$, the definition of $\eta_r$ and the properties of the pull-back.
\end{proof}
\begin{proposition}\label{prop:nijenhuis_phi}
The restriction $N_{\phi}|_{H\times H}$ is equal to the tensor $\dx \eta \otimes \xi$.
\end{proposition}

\begin{proof}
By the convergence in $\mathcal{C}^1$ topology of $\phi_r$ to $\phi$, it follows that for all $u,v$, tangent to $H$, $N_{\phi_r}(u,v) \to N_{\phi}(u,v)$ as $r\to +\infty$.
Since $H= \ker \eta$, it follows from equation \eqref{eq:eta-eta_r} that there exists $c>0$ such that $|e^r\eta_r(u)| \leqslant c\|u\|_ge^{-\frac{r}{2}}$ and $|\eta_r(v)| \leqslant c\|v\|_ge^{-\frac{r}{2}}$.
It then follows from Lemma \ref{lem:s_r-phi_r_commute} that there exists $c'>0$ such that
\begin{equation*}
\forall r \geqslant 0, \forall u,v \in \Gamma(H),\quad \|N_{\phi_r}(u,v) - \dx \eta_r(u,v)\xi^r\|_g \leqslant c' \|u\|_g\|v\|_{g_0} e^{-\frac{3r}{2}}.
\end{equation*}
The result follows by taking the limit as $r\to +\infty$.
\end{proof}

We shall now highlight the link between the canonical form $\eta$, the Carnot-Carathéodory metric $\gamma_H$, and the almost-complex structure $J_H$.

\begin{proposition}\label{prop:strictly_pseudoconvex}
The Carnot-Carathéodory metric $\gamma_H$, the exterior differential of the canonical $1$-form at infinity $\eta$ and the tensor $\phi$ are related by the equality $\gamma_H = \dx \eta(\cdot,\phi \cdot)$.
In particular, the restriction of $\gamma_H$ to $H\times H$ is given by $\dx\eta|_{H\times H}(\cdot,J_H,\cdot)$.
\end{proposition}

\begin{proof}
Fix a radially parallel $J$-orthonormal frame $\{\delr,\Jdelr,E_1,\ldots,E_{2n}\}$.
From equation \eqref{eq:wedge}, it holds that $\dx \eta$ locally reads
\begin{equation*}
\dx \eta = \sum_{k=1}^n \eta^{2k-1}\wedge \eta^{2k} = \sum_{k=1}^{n} \eta^{2k-1}\otimes \eta^{2k} - \eta^{2k}\otimes \eta^{2k-1},
\end{equation*}
while $\gamma_H$ is locally expressed by
\begin{equation*}
\gamma_H = \sum_{j=1}^{2n} \eta^j\otimes \eta^j.
\end{equation*}
Since the radially parallel orthonormal frame is chosen to be a $J$-orthonormal frame, it holds that for all $r\geqslant 0$, and for all $k\in \{1,\ldots,n\}$,
\begin{equation*}
\begin{cases}
\eta_r^{2k-1}\circ \phi_r &= - \eta_r^{2k},\\
\eta_r^{2k}\circ \phi_r &=  \eta_r^{2k-1}.
\end{cases}
\end{equation*}
Taking the limit as $r\to +\infty$ shows that for all $k\in \{1,\ldots,n\}$, $\eta^{2k-1}\circ \phi = - \eta^{2k}$ and $\eta^{2k}\circ \phi =\eta^{2k-1}$.
It follows that
\begin{equation*}
\dx \eta (\cdot,\phi\cdot) = \sum_{k=1}^n \eta^{2k-1}\otimes \eta^{2k-1}+\eta^{2k}\otimes \eta^{2k} = \gamma_H,
\end{equation*}
which concludes the proof.
\end{proof}

We are now able to prove the last of our main Theorems (we restate the hypotheses for the reader's convenience).

\begin{thmx}
\label{thm:final}
Let $(M,g,J)$ be a complete non-compact K\"ahler manifold with an essential subset $K$, such that the sectional curvature of $\bar{M\setminus K}$ is negative.
Assume that $M$ is \ALCH and \ALS of order $a$ and $b$ with $\min\{a,b\} >\frac{3}{2}$.
Then $(\dK,H,J_H)$ is a strictly pseudoconvex CR manifold of class $\mathcal{C}^1$.
\end{thmx}

\begin{proof}
Let $\eta$ be the canonical form, $\gamma_H$ be the Carnot-Carath\'eodory metric and $\phi$ be defined as in Definition \ref{dfn:phi}.
From Proposition \ref{prop:almost_complex}, $\phi$ induces a $\mathcal{C}^1$ almost complex structure $J_H = \phi|_H$ on the distribution $H=\ker\eta$, which is contact (Theorem \ref{thm:C1}).
It follows from propositions \ref{prop:J_H_is_(1,1)} and \ref{prop:nijenhuis_phi} that $\phi$ satisfies
\begin{equation*}
\forall u,v \in \Gamma(H),\quad
\begin{cases}
N_{\phi}(u,v) &= \dx \eta(u,v) \xi, \text{ and}\\
\dx \eta(\phi u, v) &= -\dx \eta (u,\phi v).
\end{cases}
\end{equation*}
It then follows from Proposition \ref{prop:integrability_conditions} that $J_H$ is integrable, and $(\dK, H,J_H)$ is a CR manifold of class $\mathcal{C}^1$.
From Proposition \ref{prop:strictly_pseudoconvex}, it holds that $\dx \eta|_{H\times H}(\cdot,J_H \cdot) =\gamma_H$, and since $\gamma_H$ is positive definite on $H$, it follows that $(\dK,H,J_H)$ is strictly pseudoconvex.
This concludes the proof.
\end{proof}

\begin{remark}
If $M$ has real dimension $4$, then the contact structure on the boundary of $M$ is of rank $2$.
Since almost complex structures (when they are at least of class $\mathcal{C}^1$) are always integrable in dimension $2$, the proof of Theorem \ref{thm:final} is thus considerably reduced in that specific case.
\end{remark}

\begin{appendices}
\numberwithin{equation}{section}
\section{Curvature computations}
\label{appendix:A}

\begin{lemma}\label{appendix:thmB1}
In the context of Proposition \ref{prop:convergence_to_alpha}:
\begin{equation*}
R^0(\delr,\Jdelr,Y_v,SY_u) = -\frac{1}{2}g(SY_u,JY_v),
\end{equation*}
\begin{equation*}
R^0(\delr,Y_v,SY_u,\Jdelr) = \frac{1}{4}g(SY_u,JY_v),
\end{equation*}
\begin{equation*}
R^0(\delr,\nabla_{Y_v}Y_u,\delr,\Jdelr) = -g(\nabla_{Y_v}Y_u,\Jdelr), \text{ and}
\end{equation*}
\begin{equation*}
R^0(\delr,Y_v,\delr,\nabla_{Y_u}\Jdelr) = -\frac{1}{4}g(Y_v,\nabla_{Y_u}\Jdelr).
\end{equation*}
\end{lemma}
\begin{proof} Since $\delr \perp SY_u$, $\delr \perp Y_v$, $\Jdelr\perp JSY_u$ and $\Jdelr \perp JY_v$, it follows that
\begin{equation*}
\begin{split}
R^0(\delr,\Jdelr,Y_v,SY_u) &= \frac{1}{4}\big( g(\delr,SY_u)g(\Jdelr,Y_v) - g(\delr,Y_v)g(\Jdelr,SY_u)\\
&\quad+ g(\delr,JSY_u)g(\Jdelr,JY_v) - g(\delr,JY_v)g(\Jdelr,JSY_u)\\
&\quad	+ 2 g(\delr,J^2\delr)g(SY_u,JY_v)\big)\\
&= \frac{1}{4}(2g(\delr,-\delr)g(SY_u,JY_v))\\
&= -\frac{1}{2}g(SY_u,JY_v).
\end{split}
\end{equation*}
Similarly,
\begin{equation*}
\begin{split}
R^0(\delr,Y_v,SY_u,\Jdelr) &= \frac{1}{4}\big( g(\delr,\Jdelr)g(Y_v,SY_u) - g(\delr,SY_u)g(Y_v,\Jdelr) \\
& \quad + g(\delr,J^2\delr)g(Y_v,JSY_u) - g(\delr,JSY_u)g(Y_v,J^2\delr)\\
&\quad +2g(\delr,JY_v)g(\Jdelr,SY_u)\big)\\
&= \frac{1}{4}(-g(Y_v,JSY_u))\\
&= \frac{1}{4}g(SY_u,JY_v).
\end{split}
\end{equation*}
In addition, since $J$ is skew-symmetric,
\begin{equation*}
\begin{split}
R^0(\delr,\nabla_{Y_v}Y_u,\delr,\Jdelr) &= \frac{1}{4}\big( g(\delr,\Jdelr)g(\nabla_{Y_v}Y_u,\delr) - g(\delr,\delr)g(\nabla_{Y_v}Y_u,\Jdelr)\\
&\quad +g(\delr,J^2\delr)g(\nabla_{Y_v}Y_u,\Jdelr) - g(\delr,\Jdelr)g(\nabla_{Y_v}Y_u,\Jdelr) \\
&\quad + 2g(\delr,J\nabla_{Y_v}Y_u)g(\Jdelr,\Jdelr)\big)\\
&= \frac{1}{4}\big( -g(\nabla_{Y_v}Y_u,\Jdelr) -g(\nabla_{Y_v}Y_u,\Jdelr) - 2g(\nabla_{Y_v}Y_u,\Jdelr)\big)\\
&= -g(\nabla_{Y_v}Y_u,\Jdelr).
\end{split}
\end{equation*}
Finally, since $J$ is skew-symmetric and parallel, it holds that
\begin{equation*}
\begin{split}
R^0(\delr,Y_v,\delr,\nabla_{Y_u}\Jdelr) &= \frac{1}{4}\big( g(\delr,\nabla_{Y_u}\Jdelr)g(Y_v,\delr)-g(\delr,\delr) g(Y_v,\nabla_{Y_u}\Jdelr) \\
& \quad  + g(\delr,J\nabla_{Y_u}\Jdelr)g(Y_v,\Jdelr)-g(\delr,\Jdelr)g(Y_v,\nabla_{Y_u}\Jdelr)\\
&\quad + 2g(\delr,JY_v)g(\nabla_{Y_u}\Jdelr,\Jdelr)\big)\\
&= -\frac{1}{4}g(\nabla_{Y_u}\Jdelr,Y_v) + \frac{1}{2}g(\delr,JY_v)g(\nabla_{Y_u}\Jdelr,\Jdelr).
\end{split}
\end{equation*}
To conclude, note that since $\|\Jdelr\|_g^2 =1$ is constant, then $g(\nabla_{Y_u}\Jdelr,\Jdelr) = 0$.
\end{proof}

\begin{lemma}\label{appendix:B2}
In the context of Proposition \ref{prop:convergence_to_alphaj}:
\begin{equation*}
R^0(\delr,E_j,Y_v,SY_u)  = \frac{1}{4} g(SY_u,\Jdelr)g(Y_v,\JE_j) -\frac{1}{4} g(SY_u,\JE_j)g(Y_v,\Jdelr),
\end{equation*}
\begin{equation*}
R^0(\delr,Y_v,SY_u,E_j) = \frac{1}{4}g(SY_u,\Jdelr)g(Y_v,\JE_j) + \frac{1}{2}g(SY_u,\JE_j)g(Y_v,\Jdelr),
\end{equation*}
\begin{equation*}
R^0(\delr,\nabla_{Y_v}Y_u,\delr,E_j) = -\frac{1}{4}g(\nabla_{Y_v}Y_u,E_j), \text{ and}
\end{equation*}
\begin{equation*}
R^0(\delr,Y_v,\delr,\nabla_{Y_u}E_j) = -\frac{1}{4}g(Y_v,\nabla_{Y_u}E_j) - \frac{3}{4}g(SY_u,\JE_j)g(Y_v,\Jdelr).
\end{equation*}
\end{lemma}
\begin{proof}
Since $\delr\perp SY_u$, $\delr \perp Y_v$ and $\delr \perp \JE_j$, it holds that
\begin{equation*}
\begin{split}
R^0(\delr,E_j,Y_v,SY_u) &= \frac{1}{4}\big( g(\delr,SY_u)g(E_j,Y_v) - g(\delr,Y_v)g(E_j,SY_u)\\
& \quad + g(\delr,JSY_u)g(E_j,JY_v) - g(\delr,JY_v)g(E_j,JSY_u)\\
& \quad + 2g(\delr,\JE_j)g(SY_u,JY_v)\big) \\
&= \frac{1}{4}(g(\delr,JSY_u)g(E_j,JY_v) - g(\delr,Y_v)g(E_j,JSY_u))\\
&= \frac{1}{4}g(SY_u,\Jdelr) - \frac{1}{4}g(SY_u,\JE_j).
\end{split}
\end{equation*}
Similarly, it holds that
\begin{equation*}
\begin{split}
R^0(\delr,Y_v,SY_u,E_j) &= \frac{1}{4}\big( g(\delr,E_j)g(Y_v,SY_u) - g(\delr,SY_u)g(Y_v,E_j)\\
& \quad + g(\delr,\JE_j)g(Y_v,JSY_u) - g(\delr,JSY_u)g(Y_v,\JE_j)\\
& \quad + 2g(\delr,JY_v)g(E_j,JSY_u)\big) \\
&= -\frac{1}{4}g(\delr,JSY_u)g(Y_v,\JE_j) + \frac{1}{2}g(\delr,JY_v)g(E_j,JSY_u)\\
&= \frac{1}{4}g(SY_u,\Jdelr)g(Y_v,\JE_j) +\frac{1}{2}g(SY_u,\JE_j)g(Y_v,\JE_j).
\end{split}
\end{equation*}
Since in addition $\delr \perp \Jdelr$, it follows that
\begin{equation*}
\begin{split}
R^0(\delr,\nabla_{Y_v}Y_u,\delr,E_j) &= \frac{1}{4}\big( g(\delr,E_j)g(\nabla_{Y_v}Y_u,\delr) - g(\delr,\delr)g(\nabla_{Y_v}Y_u,E_j)\\
& \quad + g(\delr,\JE_j)g(\nabla_{Y_v}Y_u,\Jdelr) - g(\delr,\Jdelr)g(\nabla_{Y_v}Y_u,\JE_j)\\
& \quad + 2g(\delr,J\nabla_{Y_v}Y_u)g(E_j,\Jdelr)\big) \\
&= -\frac{1}{4}g(\nabla_{Y_v}Y_u,E_j).
\end{split}
\end{equation*}
Finally, since $J$ is parallel and $g(\nabla_{Y_u}\JE_j,\delr) = -g(SY_u,\JE_j)$, it holds that
\begin{equation*}
\begin{split}
R^0(\delr,Y_v,\delr,\nabla_{Y_u}E_j) &= \frac{1}{4}\big( g(\delr,\nabla_{Y_u}E_j)g(Y_v,\delr) - g(\delr,\delr)g(Y_v,\nabla_{Y_u}E_j)\\
& \quad + g(\delr,J\nabla_{Y_u}E_j)g(Y_v,\Jdelr) - g(\delr,\Jdelr)g(Y_v,J\nabla_{Y_u}E_j)\\
& \quad + 2g(\delr,JY_v)g(\nabla_{Y_u}E_j,\Jdelr)\big) \\
&= -\frac{1}{4} g(Y_v,\nabla_{Y_u}E_j) +\frac{1}{4}g(\delr,J\nabla_{Y_u}E_j)g(Y_v,\Jdelr)\\
&\quad + \frac{1}{2}g(\delr,JY_v)g(\nabla_{Y_v}E_j,\Jdelr)\\
&= -\frac{1}{4}g(Y_v,\nabla_{Y_u}E_j) + \frac{3}{4}g(\nabla_{Y_u}\JE_j,\delr)g(Y_v,\Jdelr)\\
&= -\frac{1}{4}g(Y_v,\nabla_{Y_u}E_j) - \frac{3}{4}g(SY_u,\JE_j)g(Y_v,\Jdelr).
\end{split}
\end{equation*}
\end{proof}

\end{appendices}

\printbibliography

\end{document}